\documentclass[11pt,a4paper]{article}
\usepackage{amsfonts,amssymb,mathrsfs}
\usepackage{amsopn}
\usepackage{amsmath,amsthm}
\usepackage{psfrag}

\DeclareMathOperator{\C}{\mathbb{C}}
\DeclareMathOperator{\D}{\mathbb{D}}

\DeclareMathOperator{\dist}{dist}

\DeclareMathOperator{\N}{\mathbb{N}}

\DeclareMathOperator{\Q}{\mathbb{Q}}

\DeclareMathOperator{\rel}{\textup{rel}}
\DeclareMathOperator{\R}{\mathbb{R}}

\DeclareMathOperator{\Z}{\mathbb{Z}}
\DeclareMathAlphabet{\mathpzc}{OT1}{pzc}{m}{it}

\newcommand{\bA}{\boldsymbol{A}}
\newcommand{\ba}{\boldsymbol{a}}
\newcommand{\be}{\boldsymbol{e}}
\newcommand{\bR}{\boldsymbol{R}}
\newcommand{\br}{\boldsymbol{r}}
\newcommand{\bx}{\boldsymbol{x}}
\newcommand{\bX}{\boldsymbol{X}}
\newcommand{\by}{\boldsymbol{y}}
\newcommand{\bz}{\boldsymbol{z}}
\newcommand{\Dc}{\mathscr{D}}
\newcommand{\Hc}{\mathcal{H}}
\newcommand{\Pc}{\mathcal{P}}

\theoremstyle{plain}
\newtheorem{thm}{Theorem}

\newtheorem{lem}[thm]{Lemma}
\newtheorem*{cor}{Corollary}
\newtheorem{prop}[thm]{Proposition}

\theoremstyle{definition}

\theoremstyle{remark}
\newtheorem*{rem}{Remark}
\newtheorem*{ack}{Acknowledgements}

\title{Hardy and Lieb-Thirring inequalities for anyons\thanks{This work was partially supported by the Danish Council for Independent Research.}}

\author{Douglas Lundholm and Jan Philip Solovej}

\date{\scriptsize{Department of Mathematical Sciences, University of Copenhagen\\ Universitetsparken 5, DK-2100 Copenhagen \O, Denmark}}

\begin{document}

\maketitle

\begin{abstract}
	We consider the many-particle quantum mechanics of \emph{anyons},
	i.e. identical particles in two space dimensions 
	with a continuous statistics parameter $\alpha \in [0,1]$ 
	ranging from bosons ($\alpha=0$) to fermions ($\alpha=1$).
	We prove a (magnetic) Hardy inequality for anyons,
	which in the case that $\alpha$ is an odd numerator fraction 
	implies a local exclusion principle for the kinetic energy of
	such anyons.
	From this result, and motivated by Dyson and Lenard's 
	original approach to the stability of fermionic matter 
	in three dimensions, we prove
	a Lieb-Thirring inequality for these types of anyons.
\end{abstract}

\setlength\arraycolsep{2pt}
\def\arraystretch{1.2}

\section{Introduction}

	The concept of identical particles and associated particle statistics 
	lies at the foundations of quantum mechanics.
	It arises as a consequence of the non-observability of particle interchange
	and the fact that states in quantum mechanics are represented by rays
	in a complex Hilbert space, i.e. only determined up to a complex phase.
	A quantum mechanical state describing $N$ distinguishable particles\footnote{We 
	will for simplicity restrict to \emph{scalar} non-relativistic particles,
	i.e. point particles without internal symmetries and spin.} 
	moving in $\R^d$ is represented by an \emph{$N$-particle wave function},
	i.e. a square-integrable complex-valued function 
	$u \in L^2(\R^{dN})$ defined on $N$ copies of $\R^d$,
	or equivalently, by an element of the tensor product space
	$\bigotimes^N L^2(\R^d)$ 
	derived from the one-particle Hilbert space $L^2(\R^d)$.
	Upon restricting to \emph{identical} particles, 
	the freedom of choice of particle statistics
	stems from the fact that only the amplitude $|u(x)|$ of the wave function
	--- describing (the square root of) 
	the probability density for measuring the
	specific configuration 
	$x = (\bx_1,\bx_2,\ldots,\bx_N)$ of particle positions $\bx_j \in \R^d$
	--- is observable, but not the exact phase $u(x)/|u(x)|$.
	Hence, since there should be no observable difference between 
	the particle configuration 
	$x = (\ldots,\bx_j,\ldots,\bx_k,\ldots)$
	and the one
	$x' = (\ldots,\bx_k,\ldots,\bx_j,\ldots)$,
	with particles $j$ and $k$ interchanged, 
	the amplitudes must be the same, but the phase may differ, 
	as expressed by
	\begin{equation} \label{particle-gluing-condition}
		u(x') = e^{i\alpha\pi} u(x),
	\end{equation}
	with $\alpha \in [0,2)$.
	In three and higher dimensions one finds that the only two 
	possibilites for such a 
	phase are $e^{i\alpha\pi} = \pm 1$, 
	corresponding to \emph{bosons} (such as photons) with the plus sign,
	and \emph{fermions} (such as electrons) with the minus sign.
	It is in fact sufficient to consider the permutation group $S_N$ 
	acting on the $N$-particle Hilbert space $L^2(\R^{dN})$, 
	and the wave functions describing identical bosons respectively fermions
	are then given by the completely symmetric resp. antisymmetric 
	$N$-particle wave functions,
	which in the latter case can be 
	represented by 
	the Hilbert space $\bigwedge^N L^2(\R^d)$.
	
	However, in two dimensions the statistics parameter $\alpha$ 
	can be taken to be \emph{any} real number in the interval $[0,2)$
	(or $(-1,1]$, by periodicity).
	Again, bosons correspond to $\alpha=0$ and fermions to $\alpha=1$,
	while for a general choice of $\alpha$ the corresponding particles are
	simply called \emph{anyons} 
	(or, historically, just 
	``particles obeying \emph{intermediate} or \emph{fractional statistics}''
	\cite{Streater-Wilde:70,Leinaas-Myrheim:77,Goldin-Menikoff-Sharp:81,Wilczek:82}).
	This discrepancy between two and higher spatial dimensions is directly
	related to the fact that a punctured plane is not simply connected,
	while $\R^d \setminus \{0\}$ \emph{is}, for $d \ge 3$.
	Hence, for particles confined to the plane one has to consider \emph{continuous} 
	interchanges of particles, forming a loop in the configuration space
	(together with the possibility of enclosing other particles in that loop),
	and the permutation group symmetry is replaced by the \emph{braid group} $B_N$,
	whose one-dimensional unitary representations determine 
	the choice of statistics for identical anyons through the phase $e^{i\alpha\pi}$.
	For a convenient and rigorous treatment of anyons, 
	one can model them as identical bosons or fermions in the plane, 
	but with a magnetic interaction of Aharonov-Bohm type between each
	pair of particles, giving rise to the correct statistics phase
	as the particles encircle each other.
	The quantum mechanical momenta of the particles, 
	which for bosons and fermions are simply given by the gradients 
	$-i\nabla_j u$ w.r.t. the particle positions $\bx_j$,
	will then be replaced by covariant (magnetic) derivatives 
	$D_j u$ (see \eqref{covariant_derivative} below).
	
	Fermions in any dimension are special, 
	since they satisfy the so-called \emph{Pauli exclusion principle}.
	Namely, because of the antisymmetry of the wave function,
	no two particles can occupy the exact same state,
	expressed simply with the help of the wedge product as 
	$u_0 \wedge u_0 = 0$ for any one-particle state $u_0 \in L^2(\R^d)$.
	An important and non-trivial consequence of this is
	the celebrated \emph{Lieb-Thirring inequality}, 
	which given a scalar potential $V$ on $\R^d$
	can be summarized as 
	\begin{multline} \label{Lieb-Thirring}
		\sum_{j=1}^N \int_{\R^{dN}} \left( |\nabla_j u|^2 + V(\bx_j)|u|^2 \right) dx \\
		\ \ge \ -\sum_{k=0}^{N-1} |\lambda_k(h)|
		\ \ge \ -C^{\textup{LT}}_d \int_{\R^d} |V_-(\bx)|^{1 + \frac{d}{2}} \,d\bx,
	\end{multline}
	where $\lambda_k(h)$ denote the negative eigenvalues
	(ordered with decreasing magnitude)
	of the one-particle Schr\"odinger operator
	$h := -\Delta_{\R^d} + V(\bx)$ acting in $L^2(\R^d)$.
	The first inequality just expresses the fact that, because of the
	Pauli principle for fermions, the lowest possible energy 
	(l.h.s. of \eqref{Lieb-Thirring}) is obtained
	when the particles assume the 
	states corresponding to the
	lowest $N$ eigenvalues $\lambda_k(h)$,
	i.e. when $u = \bigwedge_{k=0}^{N-1} u_k$ 
	is an antisymmetrized product of those one-particle eigenfunctions $u_k(h)$. 
	The second inequality holds uniformly in $N$ 
	and concerns the trace over the negative 
	spectrum of the one-particle operator $h$.
	Ever since the first proof of \eqref{Lieb-Thirring} 
	in 1975 by Lieb and Thirring \cite{Lieb-Thirring:75},
	who used this result for a simplified proof of 
	stability of fermionic matter (see also \cite{Lieb-Seiringer:10}), 
	there has been a lot of activity in the mathematical community 
	aiming to generalize
	this type of spectral estimate for one-particle operators
	in various directions.
	Furthermore, the Lieb-Thirring inequality \eqref{Lieb-Thirring} 
	(disregarding the intermediate sum over eigenvalues) 
	is equivalent to the \emph{kinetic energy inequality}
	\begin{equation} \label{kinetic-energy-inequality}
		T := \sum_{j=1}^N \int_{\R^{dN}} |\nabla_j u|^2 \,dx
		\ \ge \ C^{\textup{K}}_d \int_{\R^d} \rho(\bx)^{1 + \frac{2}{d}} \,d\bx,
	\end{equation}
	which can be interpreted as 
	a strong form of the uncertainty principle for fer\-mi\-ons,
	because of the way it bounds the total kinetic energy $T$
	in terms of the one-particle density 
	\begin{equation} \label{one-particle-density}
		\rho(\bx) := \sum_{j=1}^N \int_{\R^{d(N-1)}} 
		|u(\bx_1,\ldots,\bx_j = \bx,\ldots,\bx_N)|^2 
		\prod_{k \neq j} d\bx_k
	\end{equation}
	of the wave function 
	(always assumed to be normalized to $\int_{\R^{dN}} |u|^2 \,dx = 1$).
	
	Bosons, on the other hand, do not satisfy any exclusion principle.
	They can all be put in the same state,
	e.g. $u = u_0 \otimes \ldots \otimes u_0$,
	and hence cannot be expected to satisfy the inequalities
	\eqref{Lieb-Thirring} or \eqref{kinetic-energy-inequality}.
	In fact, the best we can do is to treat them as $N$ 
	copies of a single 
	particle satisfying \eqref{Lieb-Thirring},
	and hence the above inequalities hold only in the weaker form
	$$
		\sum_{j=1}^N \int_{\R^{dN}} \left( |\nabla_j u|^2 + V(\bx_j)|u|^2 \right) \,dx
		\ \ge \ -N \,C^{\textup{LT}}_d \int_{\R^d} |V_-(\bx)|^{1 + \frac{d}{2}} \,d\bx,
	$$
	resp. (cp. Appendix B)
	$$
		\sum_{j=1}^N \int_{\R^{dN}} |\nabla_j u|^2 \,dx
		\ \ge \ \frac{C^{\textup{K}}_d}{N^{2/d}} \int_{\R^d} \rho(\bx)^{1 + \frac{2}{d}} \,d\bx,
	$$
	which now only encodes the uncertainty principle, 
	without the extra gain in \eqref{kinetic-energy-inequality}
	due to statistics.
	These bosonic inequalities become trivial as $N \to \infty$.
	
	In contrast, not much has been known about the 
	spectral and statistical properties
	of many anyons for $0 < |\alpha| < 1$. 
	Not even the ground state energy of an otherwise non-interacting 
	gas of anyons has been rigorously estimated
	(note that this is trivial in the case of bosons, 
	and a simple exercise in the case of fermions).
	Even though we live in a three-dimensional world, there are situations
	where anyons are believed to describe,
	in the form of quasi-particles,
	the excitations in effective two-dimensional highly correlated systems 
	(see e.g. \cite{Khare:05}).
	Analyzing the thermodynamic properties of a gas of anyons may 
	therefore be of importance in understanding such systems. 
	Although the free anyon gas, which we consider here, is much simpler 
	than the effective systems mentioned above, we consider investigating 
	its thermodynamic properties a first step in such an analysis.

	The problem 
	of understanding the ground state properties of the anyon gas 
	has been attacked by many authors 
	through various approximations (see the references in the books and reviews
	\cite{Froehlich:90,Khare:05,Lerda:92,Myrheim:99,Wilczek:90})\footnote{
	We should also point out \cite{Baker_et_al:93} where a certain
	class of anyons with a strong hard-core repulsion is considered,
	as well as \cite{DFT} where the formalism is discussed.}.
	The main difficulty 
	lies in the fact that many-anyon wave functions
	cannot be simply related to one-particle wave functions in the same way
	as for bosons and fermions. 
	As will be seen explicitly below, the Hamiltonian operator 
	$H_0 = \sum_{j=1}^N D_j^2$
	describing the kinetic energy of $N$ free
	anyons is not just a free Laplacian acting on totally symmetric or antisymmetric
	wave functions, but involves long-range magnetic interactions between
	all the particles.
	In particular, we cannot reduce our study to the relatively simple case 
	of a one-particle Schr\"odinger operator $h$.
	The aim of the present paper is to address this gap in 
	knowledge concerning intermediate anyon statistics, 
	as well as the current lack of techniques to study 
	the spectral theory of such quantum mechanical systems.
	In particular, we develop a technique 
	for proving Lieb-Thirring inequalities for interacting systems,
	resulting in Theorem \ref{thm:LT-anyon-intro} below for anyons.

	The operator $D_j$ given explicitly in \eqref{covariant_derivative} 
	is singular when particle $j$ coincides with one of the other particles.
	The operator $H_0$ is therefore defined as the Friedrichs extension
	from smooth functions vanishing when particles coincide.
	This might seem to imply a hard-core condition on anyons,
	but we have recently shown in \cite{Lundholm-Solovej:extended}
	that it actually does not, since it also corresponds to the \emph{maximal}
	extension of the energy form
	$u \mapsto \int_{\R^{2N}} \sum_j |D_j u|^2 \,dx$.
	
	Our first main result concerns a magnetic many-particle Hardy inequality
	for $N$ anyons, which when considered on 
	the full two-dimensional plane $\R^2$
	reads:
	\begin{equation} \label{many-anyon-Hardy-simplified}
		\sum_{j=1}^N \int_{\R^{2N}} |D_j u|^2 \,dx
		\ \ge \ \frac{4 C_{\alpha,N}^2}{N} 
			\sum_{i<j} \int_{\R^{2N}} \frac{|u|^2}{|\bx_i - \bx_j|^2} \,dx,
	\end{equation}
	with the statistics-dependent constant
	\begin{equation} \label{statistics-constant}
		C_{\alpha,N} := \min_{p=0,1,\ldots,N-2} \ \min_{q \in \Z} |(2p+1)\alpha - 2q|.
	\end{equation}
	A stronger form of \eqref{many-anyon-Hardy-simplified},
	given as Theorem \ref{thm:many-anyon-Hardy} below, and
	valid for any convex subdomain $\Omega \subseteq \R^2$,
	is shown to produce a local form of 
	\emph{Pauli's exclusion principle for anyons}
	(given as Lemma \ref{lem:energy_in_ball} below),
	whose strength depends on the large-$N$ behavior of the constant 
	$C_{\alpha,N}$.
	Although $C_{\alpha,N}$ 
	is clearly non-zero for all $N$
	whenever $\alpha$ is irrational, we find that
	$\inf_{N \in \N} C_{\alpha,N} = 0$, unless 
	$\alpha = \frac{\mu}{\nu}$
	with $\mu$ and $\nu$ relatively prime integers and $\mu$ odd, 
	in which case we shall see that
	$\inf_{N \in \N} C_{\alpha,N} = \frac{1}{\nu}$.

	For such \emph{odd numerator fractions} $\alpha = \frac{\mu}{\nu}$, 
	the energy given by the local Pauli exclusion principle 
	for arbitrary numbers of particles
	is of a similar form as for fermions, 
	but with an extra factor $\frac{1}{\nu^2}$ 
	(depending wildly on the statistics parameter).
	Our second main result is that this local bound is sufficient to
	produce a Lieb-Thirring inequality for these types of anyons.
	In our approach we have been
	inspired by Dyson and Lenard's original
	proof of the stability of fermionic matter in
	three dimensions \cite{Dyson-Lenard:67,D,L}
	(from 1967, before the advent of the Lieb-Thirring inequality), 
	in which the only place where the Pauli
	principle came in was through such a local bound for the energy.
	
\begin{thm} \label{thm:LT-anyon-intro}
	For a normalized wave function $u$ of $N$ anyons on $\R^2$,
	i.e. a completely symmetric function $u \in L^2((\R^2)^N)$
	in the quadratic form domain of $H_0 = \sum_{j=1}^N D_j^2$,
	with odd-fractional statistics parameter $\alpha = \frac{\mu}{\nu}$
	(i.e. a reduced fraction with $\mu$ odd), 
	we have the kinetic energy inequality
	\begin{equation}
		T := \sum_{j=1}^N \int_{\R^{2N}} |D_j u|^2 \,dx
		\ \ge \ \frac{1}{\nu^2} \,C_{\textup{K}} \int_{\R^2} \rho(\bx)^2 \,d\bx,
	\end{equation}
	where $\rho$ is the corresponding one-particle density \eqref{one-particle-density},
	and hence the Lieb-Thirring inequality
	\begin{equation}
		\sum_{j=1}^N \int_{\R^{2N}} \left( |D_j u|^2 + V(\bx_j)|u|^2 \right) \,dx
		\ \ge \ -\nu^2 \,C_{\textup{LT}} \int_{\R^2} |V_-(\bx)|^2 \,d\bx,
	\end{equation}
	for any real-valued potential $V$ on $\R^2$.
	Here $C_{\textup{K}}$ and $C_{\textup{LT}}$ 
	are universal positive constants 
	that can be given explicitly.
\end{thm}
	
	In particular, this implies that the total kinetic energy $T$ 
	per unit area
	for a non-interacting (apart from the statistical interaction) 
	gas of anyons with odd-fractional statistics $\alpha = \frac{\mu}{\nu}$, 
	confined to an area $A$,
	is bounded below by
	\begin{equation} \label{anyon-gas}
		\frac{T}{A} \ \ge \ C_{\textup{K}} \frac{\bar{\rho}^2}{\nu^2},
	\end{equation}
	where $\bar{\rho} := N/A$ is the average density of the gas.
	Our assumption on alpha being an odd numerator rational may 
	look peculiar, and could of course be an artifact of our proof, 
	however it certainly raises the very interesting question whether 
	it is necessary or not.
	Other implications of this Lieb-Thirring inequality,
	such as for interacting anyon gases,
	as well as this question of sharpness 
	concerning even-fractional and irrational statistics 
	and its physical interpretation,
	will be discussed elsewhere
	\cite{Lundholm-Solovej:exclusion}.
	We have also extended
	some of our techniques to the case of
	identical particles in one dimension
	\cite{Lundholm-Solovej:extended}.

	The structure of this paper is as follows.
	In Section 2 we fix the notation and briefly recall the 
	general theory of particle statistics. 
	Our fundamental many-particle Hardy inequality for anyons is proven in Section 3
	based on a pairwise relative parameterization of the configuration 
	space, combined with a local magnetic Hardy inequality which 
	takes into account
	the underlying symmetry between the particles.
	Section 4 concerns the local gain in energy 
	following from these Hardy inequalities,
	which we refer to as 
	a local Pauli exclusion principle for anyons
	due to its direct similarity with the 
	corresponding local gain in energy for fermions.
	In Section 5 we use this local gain to prove 
	the Lieb-Thirring inequality for anyons.
	In the appendices we have placed some suggestions for improvements
	of the local energy bound, 
	as well as a proof of a Lieb-Thirring inequality
	on cubes with Neumann boundary conditions.
	
	We emphasize that, although a Lieb-Thirring-type inequality for anyons
	could perhaps have been anticipated based on physical grounds
	(at least for some values of the statistics parameter $\alpha$),
	this is from a purely mathematical perspective a highly non-trivial
	extension of the usual Lieb-Thirring inequality \eqref{Lieb-Thirring}
	since the relevant operator
	$$
		H = \sum_{j=1}^N \left( 
			D_j(x)^2
			+ V(\bx_j) \right),
	$$
	with $D_j(x)$ being the differential operators given explicitly 
	in \eqref{covariant_derivative} below,
	is now a strongly interacting magnetic many-particle Hamiltonian.
	Our method for proving Lieb-Thirring inequalities for interacting
	many-particle Hamiltonians
	was recently used in \cite{Frank-Seiringer}
	for a model with point interactions in three dimensions.

\vspace{10pt}

\begin{ack}
	D.L. would like to thank Oscar Andersson Forsman for discussions
	related to Appendix A.
\end{ack}

\section{Preliminaries}
	\label{sec:preliminaries}

\subsection{Particle statistics in two dimensions}

	In this subsection, which is \emph{not} a prerequisite for understanding 
	the rest of the paper,
	we give a very brief recap of 
	the general theory of particle statistics
	in $d \ge 2$ dimensions.
	For more details we refer to
	the basic reference \cite{Leinaas-Myrheim:77},
	the review articles \cite{Froehlich:90,Myrheim:99},
	and the books \cite{Khare:05,Lerda:92,Wilczek:90} on anyons.
	
	The classical configuration space of $N$ identical particles in $\R^d$
	is formally given by
	$$
		X_d^N := \left( \R^{dN} \setminus \D \right) \Big/ {S_N},
	$$
	where we have excluded all coincidences of the particles,
	i.e. the diagonals
	$$
		\D := \{ x \in \R^{dN} : \bx_j = \bx_k \ \textrm{for some $j \neq k$} \},
	$$
	and the symmetric group $S_N$ acts on 
	the $N$ copies of $\R^d$ in the obvious way.
	(The center-of-mass coordinate $\bX := \frac{1}{N}\sum_j \bx_j$ 
	can be trivially factored out, $X_d^N \cong \R^d \times X_{d,\rel}^N$,
	leaving the \emph{relative} configuration space 
	$X_{d,\rel}^N = \left( \{ x \in \R^{dN} : \bX=0 \} \setminus \D \right) \big/ {S_N}$.)
	For $d \ge 3$ the fundamental group of $X_d^N$
	is $\pi_1(X_d^N) = S_N$, whereas for $d=2$ it is 
	the braid group on $N$ strands, $\pi_1(X_2^N) = B_N$.
	Wave functions of $N$ identical particles are 
	defined as square-integrable complex-valued functions 
	on $X_d^N$ with appropriate gluing conditions
	(recall the physical requirement \eqref{particle-gluing-condition}).
	Hence, these can be viewed as sections of 
	a complex line bundle over $X_d^N$.
	There are natural flat connections on such line bundles,
	taking the trivial connection locally on $\R^{dN}$
	(note that the parallel transports could be \emph{globally}
	non-trivial as there are non-trivial transition functions 
	between local regions),
	and every such connection defines a unitary one-dimensional 
	representation of the fundamental group $\pi_1(X_d^N)$.
	For $d \ge 3$ there are only two such representations,
	the trivial one corresponding to bosons,
	and the sign on $S_N$ corresponding to fermions.
	For $d=2$ the unitary one-dimensional representations of the
	braid group are parameterized by a real number $\alpha \in [0,2)$,
	where every generator in $B_N$ corresponding to a 
	counter-clockwise interchange of two neighbouring strands is
	represented by the phase $e^{i\alpha\pi}$.
	In particular, this will imply that
	\begin{equation} \label{statistics-definition}
		(\Pc_{(p)}u)(x) = e^{i(2p+1)\alpha\pi} u(x),
	\end{equation}
	where $\Pc_{(p)}$ denotes the action of parallel transport along 
	a closed loop in $X_2^N$ corresponding to
	continuous counter-clock\-wise interchange 
	of two particles
	$\bx_j$ and $\bx_k$, with the interchange loop
	enclosing precisely 
	$0 \le p \le N-2$
	\emph{other} particles $\bx_{i_1},\ldots,\bx_{i_p}$.
	On the other hand, if a single particle $\bx_j$ is taken along
	a simple loop which encloses $p$ other particles, 
	then a phase factor $e^{i2p\alpha\pi}$ will be picked up.
	\begin{figure}[t]
		\centering
		\psfrag{p}{$p$}
		\psfrag{T_phase1}{$e^{i2p\alpha\pi}$}
		\psfrag{T_phase2}{$e^{i(2p+1)\alpha\pi}$}
		\includegraphics{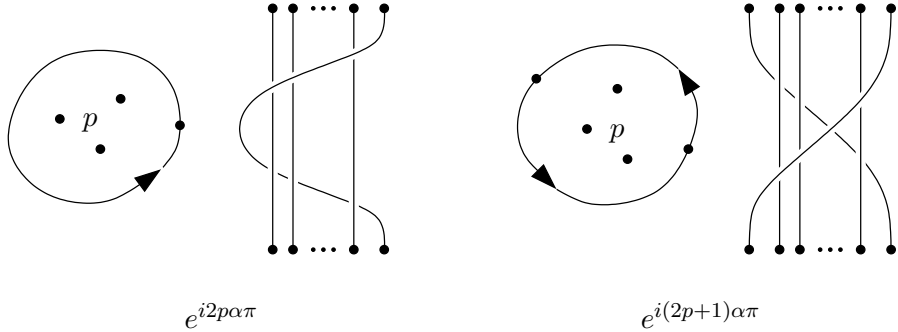}
		\caption{One- resp. two-particle interchange loops with
			corresponding braid diagrams 
			(where we can think of time as running upwards)
			and phases.}
		\label{fig:loops}
	\end{figure}
	See Figure \ref{fig:loops} for these examples,
	and for a glimpse of how the exact phases arise through 
	braid diagrams, with each elementary counter-clockwise braid contributing 
	a phase $e^{i\alpha\pi}$.
	If $\alpha = 0$ or $\alpha = 1$ the representations
	depend only on the permutations of the particles and not on the
	braid, and we are back to the case of bosons and fermions.

	We shall denote the Hilbert space of $N$-particle wave functions 
	with a given statistics parameter $\alpha$ by $\Hc_\alpha^N$.
	In the case of bosons or fermions we may identify such sections
	of line bundles with complex-valued functions on $\R^{dN}$ 
	that are either totally symmetric or totally antisymmetric, 
	and hence we obtain 
	the usual spaces of $N$-boson, resp. $N$-fermion wave functions.

	In general, the line bundles corresponding to different $\alpha$
	are topologically equivalent,
	but not geometrically equivalent. The map
	(singular gauge transformation)
	\begin{equation} \label{gauge-transformation-u}
		u(x) \quad \mapsto \quad
		\prod_{k < l} e^{i(\alpha-\alpha_0)\phi_{kl}} u(x),
		\qquad \phi_{kl} := \arg \frac{\bx_k - \bx_l}{|\bx_k - \bx_l|},
	\end{equation}
	(where we pick an arbitrary real axis to identify $\R^2$ with $\C$)
	maps the line bundle corresponding to $\alpha_0$
	to the line bundle corresponding to $\alpha$.
	The natural flat connection on the $\alpha_0$-bundle
	will then not be mapped to the natural flat connection on the 
	$\alpha$-bundle, but gives rise to the non-trivial gauge potential
	\begin{equation} \label{gauge-transformation-A}
		\bA_j := -i\left( \prod_{k<l} e^{i(\alpha - \alpha_0)\phi_{kl}} \right)^{-1} \!
			\nabla_{\bx_j} \left( \prod_{k<l} e^{i(\alpha - \alpha_0)\phi_{kl}} \right).
	\end{equation}
	We can then think of the $\alpha_0$-bundle as providing 
	a reference statistic,
	and we will choose $\alpha_0$ to be either $0$ or $1$
	in order to model general 
	statistics $\alpha$ in terms of
	bosonic or fermionic wave functions.
	In this way we can model $N$-anyon wave functions as either
	totally symmetric or antisymmetric $N$-particle wave functions
	$u \in \Hc_{\alpha_0}^N \subseteq L^2(\R^{2N})$ 
	with covariant derivatives
	$$
		D_j = -i\nabla_j + \bA_j.
	$$
	The advantage of taking this viewpoint is that 
	we then can work exclusively on the redundant but uncomplicated 
	configuration space $\R^{2N}$.
	The disadvantage, of course, is that we have to deal with a long-range
	magnetic interaction potential $\bA_j$.
	In the physics literature the former viewpoint is often referred to as
	the \emph{anyon gauge}, while the latter is called the \emph{magnetic gauge}.
	Although we will mostly stick to the magnetic gauge 
	(with $\alpha_0 := 0$), 
	it is for the purpose of intuition very useful
	to have both pictures in mind.

\subsection{Notation}
	
	In order to make precise our notation, we will 
	denote by $\Dc^N_{\alpha_0,\alpha}$ the space of 
	\emph{finite kinetic energy}  
	wave functions of $N$
	anyons with reference statistic $\alpha_0 \in \{0,1\}$
	and actual statistics parameter $\alpha$, 
	i.e. the totally symmetric/antisymmetric
	functions $u \in \Hc^N_{\alpha_0} \subseteq L^2(\R^{2N})$,
	where each particle is
	interacting with all the other through 
	Aharonov-Bohm magnetic potentials of strength 
	$\alpha - \alpha_0$,
	in the following precise sense.
	The covariant derivative acting w.r.t. particle $\bx_j$ 
	(following from \eqref{gauge-transformation-u}--\eqref{gauge-transformation-A}) 
	is 
	\begin{equation} \label{covariant_derivative}
		D_j := -i\nabla_j + \bA_j(\bx_j) 
		:= -i\nabla_{\bx_j} + (\alpha - \alpha_0) \sum_{k \neq j} (\bx_j - \bx_k)^{-1}I, 
	\end{equation}
	where $\ba^{-1} := \ba/|\ba|^2$
	and $\ba \mapsto \ba I$ denotes
	counter-clockwise rotation of the vector $\ba \in \R^2$
	by an angle $\pi/2$.
	We consider the semi-bounded quadratic form
	\begin{equation} \label{quadratic-form}
		u \ \mapsto \ 
		q(u) := \sum_{j=1}^N \int_{\R^{2N}} |D_j u|^2 \,dx
		= \sum_{j=1}^N \int_{\R^{2N}} \bar{u} \, D_j^2 u \,dx,
	\end{equation}
	defined initially on 
	$u \in C_0^\infty(\R^{2N} \setminus \D) \cap \Hc^N_{\alpha_0}$,
	i.e. the smooth totally symmetric/antisymmetric 
	and square-integrable functions on $\R^{2N}$
	with support away from diagonals.
	The space of finite kinetic energy 
	$N$-anyon wave functions $\Dc^N_{\alpha_0,\alpha}$ is then
	defined as the domain of the closure of this quadratic form $q$
	on $\Hc^N_{\alpha_0}$.
	There is also an associated smaller space 
	$\tilde{\Dc}^N_{\alpha_0,\alpha}$
	defined as the domain of the Friedrichs extension,
	i.e. of the self-adjoint operator $H_0 = \sum_j D_j^2$ 
	on $\Hc^N_{\alpha_0} \subseteq L^2(\R^{2N})$
	associated to the closure of the quadratic form \eqref{quadratic-form}.
	Although this definition might seem to imply a (mild) 
	\emph{hard-core condition} on anyons
	(cp. e.g. \cite{Baker_et_al:93,Loss-Fu:91}),
	we have shown in \cite{Lundholm-Solovej:extended} 
	that it actually poses no unnecessary restriction
	on wave functions since it agrees with the \emph{maximal} domain of 
	the quadratic form $q$ in \eqref{quadratic-form}. 
	Note e.g. that for $\alpha_0 = \alpha = 0$,
	\begin{equation} \label{domain_example}
		\tilde{\Dc}^N_{0,0} = H^2(\R^{2N}) \cap \Hc^N_0
		\ \subseteq \ H^1(\R^{2N}) \cap \Hc^N_0 = \Dc^N_{0,0},
	\end{equation}
	where $H^k$ denote the Sobolev spaces of 
	$k$ partial derivatives in $L^2$.
	In our proofs we will always use the denseness of the 
	smooth functions with compact support in these spaces,
	and pick such representatives without 
	taking explicit limits.
	
	In the following,
	open resp. closed balls of radius $r$ at a point $x$ will be denoted 
	$B_r(x)$, resp. $\bar{B}_r(x)$,
	and the characteristic function of a set $A$ is denoted $\chi_A$.
	Given a real-valued function or expression $f$, we define 
	the non-negative quantities
	$f_\pm := \max \{0, \pm f\}$.

\section{Hardy inequalities for anyons}

	We can mention \cite{H-O2-Laptev-Tidblom:08},
	where a many-particle Hardy inequality has been derived for anyons
	(see Theorem 2.7 in \cite{H-O2-Laptev-Tidblom:08}),
	but which is unfortunately not sufficient for our purposes\footnote{
	Note, for instance, that the corresponding Hardy constant in 
	\cite{H-O2-Laptev-Tidblom:08},
	$$
		D_{N,\alpha} = \min_{l=1,2,\ldots,N-1} \left( \frac{\min_{k \in \Z} |l\alpha - k|}{l} \right)^2,
	$$ 
	is zero for $\alpha = 1$
	(and any $\alpha \in \Q$ for $N$ large enough), 
	and that in any case $D_{N,\alpha} \lesssim N^{-2}$.
	Hardy inequalitites 
	for interactions of anyonic type have also
	been considered in \cite{Melgaard-Ouhabaz-Rozenblum:04},
	but these are of single-particle type and hence also
	do not take the underlying symmetry between particles into account,
	as well as have an unclear dependence of the corresponding constants 
	on the positions of the Aharonov-Bohm fluxes.}.
	The weakness stems from the fact that, so far,
	only single-particle movements have been taken into account,
	which only captures some of the symmetries involved
	(cp. Figure \ref{fig:loops}).
	In order to arrive at something non-trivial,
	at a minimum for the special case $\alpha=1$ of fermions, 
	we need to consider a \emph{relative} Hardy inequality 
	(cp. e.g. Lemma 4.6 in \cite{H-O2-Laptev-Tidblom:08}
	which one could view as a relative Hardy inequality for 
	a pair of fermions in any dimension).
	Also
	crucial for our approach is the following extension of a class of
	well-known two-dimensional magnetic Hardy inequalities 
	(see \cite{Laptev-Weidl:99,Balinsky:03,Melgaard-Ouhabaz-Rozenblum:04}),
	where we
	take the underlying symmetry of the wave function 
	and gauge potential into account.
	
\begin{lem}[Magnetic Hardy inequality with symmetry] 
	\label{lem:magnetic-Hardy}
	Let $\Omega = B_{R_2}(0) \setminus \bar{B}_{R_1}(0)$,
	$R_2 > R_1 \ge 0$,
	be an annular domain in $\R^2$,
	and let there be a magnetic flux $\Phi$ inside 
	$\bar{B}_{R_1}(0)$, determined on $\Omega$
	by a vector potential $\ba: \Omega \to \R^2$, 
	s.t. $\nabla \wedge \ba = 0$ on $\Omega$ and
	$\int_\Gamma \ba \cdot d\br = \Phi$ for any simple loop $\Gamma$
	in $\Omega$ enclosing $\bar{B}_{R_1}(0)$.
	Furthermore, assume that $\ba$ is antipodal-\emph{antisymmetric}, 
	i.e. $\ba(-\br) = -\ba(\br)$ for all $\br \in \Omega$,
	and let $v \in C^\infty(\Omega)$ be a function on $\Omega$
	with antipodal \emph{symmetry}, 
	$v(-\br) = v(\br)$ for all $\br \in \Omega$.
	Then
	\begin{equation} \label{magnetic-Hardy-boson}
		\int_{\Omega} |D_{\br}v|^2 \,d\br
		\ \ge \ \min_{k \in \Z} \left| \frac{\Phi}{2\pi} - 2k \right|^2 
			\int_{\Omega} \frac{|v|^2}{|\br|^2} \,d\br,
	\end{equation}
	where $D_{\br} := -i\nabla_{\br} + \ba(\br)$.
	
	Alternatively, if $v$ is antipodal-\emph{anti}symmetric,
	$v(-\br) = -v(\br)$ for all $\br \in \Omega$, then
	\begin{equation} \label{magnetic-Hardy-fermion}
		\int_{\Omega} |D_{\br}v|^2 \,d\br
		\ \ge \ \min_{k \in \Z} \left| \frac{\Phi}{2\pi} - (2k+1) \right|^2 
			\int_{\Omega} \frac{|v|^2}{|\br|^2} \,d\br.
	\end{equation}
\end{lem}
\begin{proof}
	There exists a gauge transformation 
	$v \mapsto \tilde{v} = e^{i\chi} v$ such that
	$|D_{\br}v|^2 = |(-i\nabla_{\br} + \tilde{\ba})\tilde{v}|^2$,
	where
	$\ba \mapsto \tilde{\ba}(\br) := \frac{\Phi}{2\pi}\br^{-1}I$
	on $\Omega$.
	Note that $\chi(\br)$, being the integral of a difference
	of two gauge potentials $\ba(\br)$ and $\tilde{\ba}(\br)$, 
	both \emph{antisymmetric} w.r.t. $\br \mapsto -\br$, 
	must be \emph{symmetric} under this antipodal map.
	Hence, if $v$ is antipodal-(anti)symmetric, then so is $\tilde{v}$.
	
	Now, we write the gauge-transformed l.h.s. of \eqref{magnetic-Hardy-boson} 
	in terms of polar coordinates $(r,\varphi)$,
	$$
		\int_{\Omega} |D_{\br}v|^2 \,d\br
		= \int_0^{2\pi} \!\! \int_{R_1}^{R_2} |\partial_r \tilde{v}|^2 r \,dr d\varphi
		+ \int_0^{2\pi} \!\! \int_{R_1}^{R_2} \frac{1}{r^2}\left| 
			\left(-i\partial_\varphi + \frac{\Phi}{2\pi} \right)\!\tilde{v}
			\right|^2 \! r \,dr d\varphi.
	$$
	Considering only the last term involving $\partial_\varphi$,
	and Fourier expanding $\tilde{v}(\br)$ on $\Omega$,
	\begin{equation} \label{Fourier-expansion}
		\tilde{v}(r,\varphi) = \tilde{v}(\br = r\be_1e^{\varphi I}) 
			= \frac{1}{\sqrt{2\pi}} \sum_{k \in \Z} \tilde{v}_k(r) e^{i2k\varphi}
	\end{equation}
	(note that we have here used the fact that 
	$\tilde{v}(r,\varphi+\pi) = \tilde{v}(r,\varphi)$ for all $\varphi$),
	we find
	\begin{equation} \label{Fourier-bound}
		\int_{R_1}^{R_2} \sum_{k \in \Z} 
			\left| 2k + \frac{\Phi}{2\pi} \right|^2 \frac{|\tilde{v}_k|^2}{r^2} r \,dr
		\ \ge \ 
		\min_{k \in \Z} \left| \frac{\Phi}{2\pi} - 2k \right|^2 
			\int_{\Omega} \frac{|v|^2}{|\br|^2} \,d\br,
	\end{equation}
	and hence arrive at the inequality \eqref{magnetic-Hardy-boson}.
	For the case of antipodal-anti\-sym\-met\-ric $v$, we can Fourier expand
	$\tilde{v}$ in odd powers of $e^{ik\varphi}$ and arrive at
	\eqref{magnetic-Hardy-fermion}.
\end{proof}

	Now, let us for simplicity first apply this to
	the case of only two anyons and prove
	a relative Hardy inequality for this system.

\begin{lem}[Relative two-anyon Hardy] \label{lem:two-anyon-Hardy}
	Let $\Omega$ be an open convex set in $\R^2$ and let 
	$u \in \Dc^2_{0,\alpha}$
	be a two-anyon wave function.
	Then
	\begin{equation} \label{two-anyon-Hardy}
		\int_{\Omega \circ \Omega} \left( |D_1 u|^2 + |D_2 u|^2 \right) d\bx_1 d\bx_2
		\ge 2 \min_{k \in \Z} |\alpha - 2k|^2 
			\int_{\Omega \circ \Omega} \frac{|u|^2}{|\bx_1 - \bx_2|^2} 
			\,d\bx_1 d\bx_2,
	\end{equation}
	where
	\begin{equation} \label{Omega-circ-Omega}
		\Omega \circ \Omega := \left\{ \textstyle (\bx_1,\bx_2) \in \Omega^2 :
		\frac{1}{2}|\bx_1 - \bx_2| < \dist(\frac{1}{2}(\bx_1 + \bx_2), \Omega^c) \right\}
	\end{equation}
	(in particular, $\Omega \circ \Omega \subsetneq \Omega^2$,
	unless $\Omega = \R^2$).
\end{lem}
\begin{proof}
	Let us introduce the center-of-mass $\bR := \frac{1}{2}(\bx_1 + \bx_2)$
	and the relative coordinate $\br := \frac{1}{2}(\bx_1 - \bx_2)$.
	Furthermore, let $v(\bR;\br) := u(\bR+\br,\bR-\br)$, and observe that
	the bosonic symmetry of $u \in \Hc^2_0$ 
	implies $v(\bR;-\br) = v(\bR;\br)$ for all $\bR \in \Omega$ and
	$\br \in \R^2$ s.t. $0 < |\br| < \dist(\bR,\Omega^c) =: \delta(\bR)$
	(possibly infinite).
	Then, by $\nabla_{\bR} = \nabla_1 + \nabla_2$, 
	$\nabla_{\br} = \nabla_1 - \nabla_2$, 
	the l.h.s. of \eqref{two-anyon-Hardy} 
	equals
	\begin{multline*}
		\int_\Omega \int_{B_{\delta(\bR)}(0)} \left( 
			\left| \left( -\frac{i}{2}\nabla_{\bR} -\frac{i}{2}\nabla_{\br} + \frac{\alpha}{2}\br^{-1}I \right)v \right|^2 \right. \\ 
			\shoveright{ + \left.
			\left| \left( -\frac{i}{2}\nabla_{\bR} +\frac{i}{2}\nabla_{\br} - \frac{\alpha}{2}\br^{-1}I \right)v \right|^2
			\right) 2 d\br d\bR }\\
		= \int_{\Omega \circ \Omega} |\nabla_{\bR} v|^2 d\br d\bR 
			+ \int_\Omega \int_{B_{\delta(\bR)}(0)} 
				\left| \left(-i\nabla_{\br} + \alpha\br^{-1}I \right)v \right|^2 d\br d\bR.
	\end{multline*}
	For the last integral w.r.t. $\br$ we can then apply
	Lemma \ref{lem:magnetic-Hardy}
	(where no gauge transformation is necessary in this case),
	and thus find
	$$
		\int_{\Omega \circ \Omega} \left( |D_1 u|^2 + |D_2 u|^2 \right) d\bx_1 d\bx_2
		\ \ge \ 
		\min_{k \in \Z} |\alpha - 2k|^2 
			\int_\Omega \int_{B_{\delta(\bR)}(0)}
			\frac{|u|^2}{|\br|^2} \,d\br d\bR,
	$$
	which implies \eqref{two-anyon-Hardy}.
\end{proof}

	An equivalent way to think about this result in terms of the
	anyon gauge, i.e. the
	space $\Hc^2_\alpha$, is that we can make a symmetric 
	interchange of the two anyons by rotating them around their 
	common center-of-mass.
	After half a turn ---
	already then completing a loop in the relative configuration space 
	$X_{2,\rel}^2 = (\R^2 \setminus \{0\}) \big/_{\br \leftrightarrow -\br}$ ---
	we can compare the function values, and
	the condition \eqref{statistics-definition} on parallel transport 
	tells us that we should pick up a phase $e^{i\alpha\pi}$. 
	We are therefore given a function $v(\br)$ on, say, 
	the upper half-plane with a periodic boundary condition 
	$v(-r\be_1) = e^{i\alpha\pi}v(r\be_1)$,
	and hence the Hardy inequality follows
	by a simple adaptation of 
	\eqref{Fourier-expansion}-\eqref{Fourier-bound}.
	
	This is straightforwardly extended 
	using Lemma \ref{lem:magnetic-Hardy}
	to cases when there
	are also fluxes inside the interchange loop.
	Intuitively,
	as we encircle $p$ anyons with a symmetric two-anyon
	interchange loop we pick up a phase $e^{i(2p+1)\alpha\pi}$
	and hence a corresponding Hardy inequality.
	More precisely, 
	this results in the following relative
	many-particle Hardy inequality for anyons.

\begin{thm}[Many-anyon Hardy] \label{thm:many-anyon-Hardy}
	Let $\Omega$ be an open convex set in $\R^2$ and let
	$u \in \Dc^N_{0,\alpha}$ be an $N$-anyon wave function. 
	Then
	\begin{equation} \label{many-anyon-Hardy}
		\int_{\Omega^N} \sum_{j=1}^N |D_j u|^2 \,dx
		\ge \frac{4 C_{\alpha,N}^2}{N} 
			\int_{\Omega^N} \sum_{i<j} \frac{|u|^2}{|\bx_i - \bx_j|^2} 
			\chi_{\Omega \circ \Omega}(\bx_i,\bx_j) \,dx,
	\end{equation}
	with $\Omega \circ \Omega$ defined in 
	\eqref{Omega-circ-Omega},
	and
	\begin{equation} \label{relative-anyon-constant}
		C_{\alpha,N} := \min_{p=0,1,\ldots,N-2} \ \min_{q \in \Z} |(2p+1)\alpha - 2q|.
	\end{equation}
\end{thm}
\begin{proof}
	We use that, for any $z = (\bz_j) \in \C^{dN}$,
	\begin{equation} \label{abs-rel-identity}
		\sum_{j=1}^N |\bz_j|^2 
			= \frac{1}{N} \sum_{1 \le j<k \le N} |\bz_j - \bz_k|^2
			+ \frac{1}{N} \left| \sum_{j=1}^N \bz_j \right|^2.
	\end{equation}
	Applying this identity with $\bz_j := D_ju(x)$ and dropping the second term,
	we see that the l.h.s. of \eqref{many-anyon-Hardy} 
	is bounded below by a sum of $\binom{N}{2}$ integrals of the form
	\begin{equation} \label{many-anyon-integral-split}
		\int_{\Omega^{N-2}} \int_{(\bx_j,\bx_k) \in \Omega^2}
			\left| (D_j - D_k)u \right|^2
			d\bx_j d\bx_k \prod_{\substack{l=1,\ldots,N \\ l \neq j,k}} d\bx_l.
	\end{equation}
	Hence, we consider for each fixed choice of the $N-2$ variables $(\bx_l)$
	the remaining configurations of the pair 
	$(\bx_j,\bx_k)$ on $\Omega \circ \Omega \subseteq \Omega^2$,
	which we again parameterize by a center-of-mass 
	$\bR := \frac{1}{2}(\bx_j + \bx_k)$ 
	and relative coordinate $\br := \frac{1}{2}(\bx_j - \bx_k)$. 
	For each $\bR \in \Omega$ we can then split up the final parameterization
	of $\br \in B_{\delta(\bR)}(0) \setminus \{0\}$ into annuli,
	the first annulus extending from $r=\delta_0 := 0$ to $r=\delta_1$,
	defined to be the distance from $\bR$ 
	(where there could possibly be one particle $\bx_l$)
	to the next closest particle $\bx_{l'}$
	(or possibly several particles situated on the same distance
	from $\bR$).
	The next annulus then extends from $r=\delta_1$ to the next greater 
	distance $\delta_2$ from $\bR$ to any particles, 
	and so on, until we reach the boundary
	of the domain (or infinity) at $r = \delta_M := \delta({\bR})$.
	Now, on each annulus $A_m := B_{\delta_m} \setminus \bar{B}_{\delta_{m-1}}$,
	$m=1,\ldots,M$, 
	we have
	$$
		(D_j - D_k)u = \left( -i\nabla_{\br} + \alpha \br^{-1}I + \ba(\bR;\br) \right)v,
	$$
	where $v(\bR;\br) := u(\bR+\br,\bR-\br)$
	is antipodal-symmetric in $\br$, while
	the gauge potential $\ba(\bR;\br) := \bA(\bR+\br) - \bA(\bR-\br)$
	is antipodal-antisymmetric,
	with $\bA(\bx) := \alpha \sum_{l \neq j,k} (\bx - \bx_l)^{-1}I$ 
	being the magnetic potential at $\bx \in \bR + A_m$ 
	from all the other $N-2$ particles $\bx_l$.
	Note that $\nabla_{\br} \wedge \ba = 0$ on $A_m$,
	and that $\int_\Gamma \ba \cdot d\br = 4\pi\alpha p_m$
	for any simple loop $\Gamma \subset A_m$ 
	enclosing $\bar{B}_{\delta_{m-1}}$, 
	where $p_m$ is the number of particles $\bx_l$ inside $\bar{B}_{\delta_{m-1}}$. 
	We can therefore apply Lemma \ref{lem:magnetic-Hardy},
	with the total flux $\Phi$ in the disk $\bar{B}_{\delta_{m-1}}$ 
	given by $\Phi = 2\pi\alpha(1+2p_m)$.
	Hence, 
	\begin{align*}
		\int_{A_m} \left| (D_j - D_k)u \right|^2 d\br
		&\ge 4\min_{q \in \Z} |(2p_m+1)\alpha - 2q|^2 
			\int_{A_m} \frac{|u|^2}{|\bx_j - \bx_k|^2} \,d\br \\
		&\ge 4C_{\alpha,N}^2
			\int_{A_m} \frac{|u|^2}{|\bx_j - \bx_k|^2} \,d\br,
	\end{align*}
	and proceeding similarly for all annuli $A_m$, all points $\bR$,
	and all pairs $(j,k)$, we obtain the inequality \eqref{many-anyon-Hardy}.
\end{proof}

	Note that this theorem coincides with Lemma \ref{lem:two-anyon-Hardy}
	for $N=2$, and with Theorem 2.8 in \cite{H-O2-Laptev-Tidblom:08} 
	for fermions on $\R^2$ (since $C_{\alpha=1,N} = 1$ for all $N \ge 2$).
	We also
	note that $C_{\alpha=0,N} = 0$ is the optimal constant
	for bosons since \eqref{many-anyon-Hardy} concerns the \emph{Neumann} form, 
	i.e. we could in this case (and if $\Omega$ has finite measure)
	take $u$ to be constant so that the l.h.s. is identically zero.
	If one considers the \emph{Dirichlet} form on the other hand, 
	a non-trivial many-particle Hardy-type inequality for bosons 
	or distinguishable particles in
	two dimensions (also with a constant $\sim N^{-1}$,
	although with logarithmic factors in the potentials) 
	was derived in \cite{Lundholm}.
	However, we do not see how such a bound could be used to improve
	the method in this paper.

	Concerning intermediate statistics, we have the following observation:
	
	\begin{prop} \label{prop:fractionality-constant}
		The infimum
		$$
			C_{\alpha} := \inf_{N \in \N} C_{\alpha,N}
			= \inf_{p,q \in \Z} |(2p+1)\alpha - 2q|
		$$
		equals 
		$$
			C_\alpha = \left\{ \begin{array}{ll}
				\frac{1}{\nu}, \quad & 
				\textrm{if $\alpha = \frac{\mu}{\nu}$ with $\mu \in \Z$, $\nu \in \N_+$ relatively prime and $\mu$ odd,} \\
				0 & \textrm{otherwise.}
				\end{array} \right.
		$$
		Hence, $C_\alpha$ is strictly positive whenever
		$\alpha$ is an odd numerator fraction, but zero otherwise.
	\end{prop}
	
	For an upper bound in the non-trivial cases 
	we will use the following fact:
	
	\begin{lem} \label{lem:integer-relation}
		Given $a,b \in \Z$ coprime, where $a$ is odd,
		there exist integers $x,y$ such that $x$ is odd,
		$y$ is even, and $ax + by = 1$.
	\end{lem}
	\begin{proof}
		Since $a,b$ are coprime, we can by Euclid's algorithm find
		integers $x,y$ s.t. $ax + by = 1$.
		We then use that $ax + by = a(x+kb) + b(y-ka)$ for any $k \in \Z$
		and consider the different possibilities.
		If $b$ is odd then, since $a$ is also odd, 
		either $x$ is odd and $y$ even 
		and we are done, or $x$ is even and $y$ is odd,
		but then we can choose $k$ odd and so we are done.
		In the case that $b$ is even, then either $x$ is odd and $y$ even
		and we are done, or both $x$ and $y$ are odd.
		In the latter case we can again choose $k$ odd and we are done.
	\end{proof}
	
	\begin{proof}[Proof of Proposition \ref{prop:fractionality-constant}]
		Assume that $\alpha = \frac{\mu}{\nu}$, 
		with $\mu \in \Z$, $\nu \in \N_+$ coprime, 
		and consider first the case of odd numerators:
		$\mu = 2k+1$, $k \in \Z$.
		Then
		$$
			|(2p+1)\alpha - 2q| = \frac{1}{\nu}|(2p+1)(2k+1) - 2q\nu| \ge \frac{1}{\nu},
			\quad \forall p,q \in \Z,
		$$
		i.e. $C_\alpha \ge \frac{1}{\nu}$,
		since the last absolute value is a difference between
		an odd and an even integer.
		Furthermore, Lemma \ref{lem:integer-relation} 
		guarantees the existence of 
		$p,q \in \Z$ such that 
		this absolute value is equal to one,
		hence $C_\alpha = \frac{1}{\nu}$.

		In the case of even numerator fractions, 
		i.e. $\mu = 2k$, $k \in \Z$, we have
		$|(2p+1)\alpha - 2q| = \frac{2}{\nu}|(2p+1)k - \nu q|$.
		As $\nu$ necessarily is odd, 
		we can choose $2p+1 = \nu$ and $q = k$, 
		and hence $C_\alpha = 0$.
		
		In the irrational case $\alpha \in \R \setminus \Q$,
		we can find an infinite sequence of $p,q \in \Z$ such that
		(see \cite{Scott:40})
		$$
			\left| \alpha - \frac{2q}{2p+1}\right| < \frac{1}{(2p+1)^2},
		$$
		and hence $C_\alpha = 0$.
	\end{proof}

	In Figure \ref{fig:statistics-constant}
	we have plotted a sketch of the dependence on $\alpha$ 
	of the large-$N$ constant $C_\alpha$.
	\begin{figure}[t]
		\centering
		\scalebox{0.4}{ \includegraphics{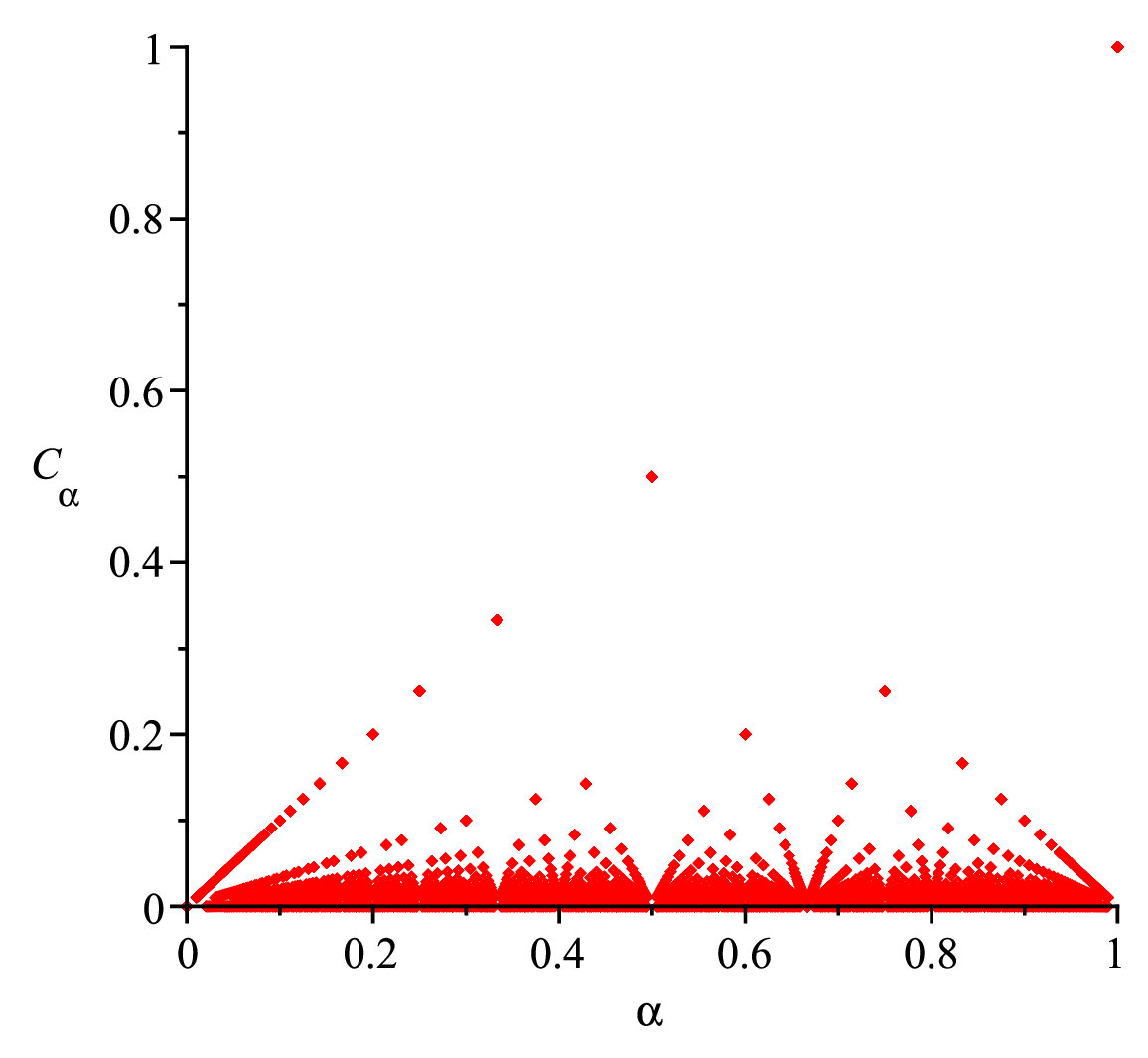} }
		\caption{A sketch of the behavior of $C_\alpha$ as a function of $\alpha$.}
		\label{fig:statistics-constant}
	\end{figure}	
	It should be emphasized that for small numbers of particles
	(relative to a fixed $\alpha$),
	the Hardy constant $C_{\alpha,N}$ is typically non-zero.
	Its graph can be obtained by cutting out wedges of slopes
	$2p+1$, $p=0,1,\ldots,N-1$, at every rational point with
	denominator $2p+1$ and an even numerator.
	In particular, $C_{\alpha,N} > 0$ for all irrational $\alpha$,
	and for even numerator fractions $\alpha = \frac{\mu}{\nu} \in (0,1)$
	the constant is strictly positive if and only if $N < \frac{\nu+3}{2}$.
	Let us also point out that we could just as well have chosen to work
	with the fermionic reference statistic, i.e. $\alpha_0 = 1$,
	for which the corresponding statistics-dependent constant is
	\begin{equation} \label{statistics-constant-beta}
		\tilde{C}_{\beta,N} := \min_{p=0,1,\ldots,N-2} \ \min_{q \in \Z} |(2p+1)\beta - (2q+1)|.
	\end{equation}
	We have here denoted the strength of the statistical interaction
	by $\beta := \alpha-1$, 
	and the values of $\beta$ for which
	this constant is bounded away from zero for all $N$
	are the fractions where
	either the numerator or denominator is even.

\section{A local Pauli exclusion principle for the kinetic energy}
	\label{sec:energy}

	Note that for odd-fractional statistics,
	i.e. for non-zero values of $C_\alpha$,
	the total constant in the many-anyon Hardy inequality
	\eqref{many-anyon-Hardy} still tends to zero like $N^{-1}$. 
	Naively\footnote{Note that the sharp large-N behavior of 
	this constant is not clear even in the fermionic case, 
	cp. \cite{H-O2-Laptev-Tidblom:08}.},
	this would be insufficient for a non-trivial bound on the
	energy of the anyon gas in the thermodynamic limit, 
	since the energy per area 
	($L^2$, say, to which the gas is confined) due to 
	Theorem \ref{thm:many-anyon-Hardy} yields
	$$
		\frac{T}{L^2} \ \ge \ \frac{1}{L^2} \cdot \frac{4C_\alpha^2}{N} 
			\cdot \frac{\binom{N}{2}}{2L^2} \cdot \int |u|^2 dx
		= \textrm{const} \cdot \frac{N-1}{N^2} \cdot \bar{\rho}^2 \to 0,
	$$
	as $N \to \infty$, with fixed density $\bar{\rho} := N/L^2$. 
	However, Theorem \ref{thm:many-anyon-Hardy} is stronger than that, and
	we can choose not to consider this global 
	gain in energy due to statistics directly,
	but rather its \emph{local} implications
	upon cutting the space up and employing Neumann boundary conditions.
	This local approach is in the spirit of Dyson and Lenard's original proof
	of the stability of matter \cite{Dyson-Lenard:67}.
	By choosing the size of such local regions appropriately,
	the energy can then be lifted to a stronger bound 
	on the full domain $\R^2$
	(see \eqref{kinetic-LT-anyons} and \eqref{gs-energy-bound} below).
	
	Central for this approach is the following version of Lemma 5
	in \cite{Dyson-Lenard:67} for anyons.
	We refer to this as a local exclusion principle for the kinetic
	energy of anyons, since it implies that $n \ge 2$ anyons
	must have positive energy and therefore cannot all occupy the
	lowest zero energy state.

\begin{lem}[Local energy / exclusion principle] \label{lem:energy_in_ball}
	Let $u \in \Dc^n_{0,\alpha}$ be a wave function of $n$ anyons 
	and let $\Omega \subseteq \R^2$ be either a disk or a square,
	with area $|\Omega|$.
	Then
	\begin{equation} \label{energy_in_ball}
		\int_{\Omega^n} \sum_{j=1}^n |D_j u|^2 \,dx
		\ \ge \ (n-1) \frac{c_\Omega C_{\alpha,n}^2}{|\Omega|} 
			\int_{\Omega^n} |u|^2 \,dx,
	\end{equation}
	where $c_\Omega$ is a constant which satisfies $c_\Omega \ge 0.169$
	for the disk, and $c_\Omega \ge 0.112$ for the square.
\end{lem}
\begin{proof}
	Note that by rescaling,
	we can in the following assume that 
	$\Omega = B_1(0)$ or $\Omega = (-1,1)^2$. 
	We shall first consider the case of the disk, 
	and then point out what needs to be changed for the square.

	Due to $\Omega \circ \Omega \subsetneq \Omega^2$,
	the bound given by the many-anyon Hardy inequality 
	\eqref{many-anyon-Hardy} is unfortunately not sufficient 
	as it stands, and we need to modify the approach in the proof of
	Theorem \ref{thm:many-anyon-Hardy} to take the whole two-particle domain 
	$\Omega^2$ into account.
	Instead of \eqref{abs-rel-identity} we shall therefore use
	\begin{equation} \label{abs-rel-identity-sum}
		\sum_{j=1}^n |\bz_j|^2 
			= \sum_{j<k} \left( 
				\frac{1-\kappa}{n-1} \left( |\bz_j|^2 + |\bz_k|^2 \right) 
				+ \frac{\kappa}{n} |\bz_j - \bz_k|^2 \right)
				+ \frac{\kappa}{n} \Big| \sum_j \bz_j \Big|^2,
	\end{equation}
	with $0 < \kappa < 1$.
	The last term is again thrown away while the middle one
	is employed as in the proof of Theorem \ref{thm:many-anyon-Hardy}
	to produce a Hardy potential in terms of relative coordinates
	$\br$ and $\bR$. However, for the first two terms we instead use that
	$|D_j u|^2 \ge \big| \nabla_j |u| \big|$ (diamagnetic inequality).
	Ignoring $\kappa$ for a second,
	we are hence interested in the infimum of the ratio
	\begin{equation} \label{minimize-energy}
		\int_{\Omega^2} 
		\left( |\nabla_{\bx_1} u|^2 + |\nabla_{\bx_2} u|^2 
			+ \frac{C_{\alpha,n}^2}{|\br|^2} \chi_{B_{\delta(\bR)}(0)}(\br) |u|^2 \right) d\bx_1 d\bx_2
		\Bigg/ \int_{\Omega^2} |u|^2 \,d\bx_1 d\bx_2
	\end{equation}
	over $u \in H^1(\Omega^2)$,
	which is certainly greater than 
	the lowest eigenvalue of the Schr\"odinger operator 
	($c := C_{\alpha,n} \neq 0$ in the following)
	\begin{equation} \label{Neumann-operator}
		H := -\Delta_{\Omega^2}^\mathcal{N} + f,
		\qquad
		f(\bx_1,\bx_2) := c^2g(|\bR|,|\br|),
	\end{equation}
	on $\Omega^2$
	defined with Neumann boundary conditions on $\partial \Omega^2$,
	with
	$$
		g(R,r) := \left\{ \begin{array}{lll} 
			\delta^{-2}(1-\hat{R})^{-2}, \quad & R \le \hat{R}, \ & r \le \delta(1-R), \\
			r^{-2}, & R \le \hat{R}, & \delta(1-R) < r < 1-R, \\
			0, & R > \hat{R},
			\end{array} \right.
	$$
	for some fixed cut-off parameters 
	$0 < \delta,\hat{R} < 1$, to be optimized over later.
	
	Now, denoting by $P$ the projection onto the constant function
	$u_0(x) := |\Omega^2|^{-\frac{1}{2}} = \pi^{-1}$, and $Q := 1-P$, 
	we have $(-\Delta_{\Omega^2}^\mathcal{N})P = 0$, and for the 
	first non-zero Neumann eigenvalue 
	$\lambda_1 = \lambda_1(-\Delta^\mathcal{N})$,
	$$
		(-\Delta_{\Omega^2}^\mathcal{N})Q 
		\ge \lambda_1(-\Delta_{\Omega^2}^\mathcal{N}) \,Q
		= \lambda_1(-\Delta_{\Omega}^\mathcal{N}) \,Q
		= \xi^2 Q,
	$$
	where $\xi \approx 1.8412$ denotes the first zero of the
	derivative of the Bessel function $J_1$.
	Furthermore, we have since
	\begin{multline*}
		\langle u, (PfQ + QfP)u \rangle 
		= \langle f^{\frac{1}{2}}Pu, f^{\frac{1}{2}}Qu \rangle
		+ \langle f^{\frac{1}{2}}Qu, f^{\frac{1}{2}}Pu \rangle \\
		\le 2 \|f^{\frac{1}{2}}Pu\| \|f^{\frac{1}{2}}Qu\| 
		\le \mu \|f^{\frac{1}{2}}Pu\|^2 + \frac{1}{\mu} \|f^{\frac{1}{2}}Qu\|^2
		= \langle u, (\mu PfP + \mu^{-1} QfQ) u\rangle,
	\end{multline*}
	for $u \in L^2(\Omega^2)$ and $\mu > 0$, that
	$$
		f = (P+Q)f(P+Q) \ge (1-\mu)PfP + (1-\mu^{-1})QfQ.
	$$
	These operators are estimated according to 
	$\|QfQ\| \le \|f\|_\infty = \frac{c^2}{\delta^2 (1-\hat{R})^2}$, and
	\begin{align*}
		\|PfP\| &= \int_{\Omega^2} \bar{u}_0 f u_0 \,d\bx_1 d\bx_2 
		= \frac{1}{\pi^2} \int_{\Omega^2} c^2 g(R,r) \,2 d\br d\bR \\
		&= \frac{2(2\pi)^2 c^2}{\pi^2} \int_{0}^{\hat{R}} \left(
			\int_{0}^{\delta(1-R)} \delta^{-2}(1-\hat{R})^{-2} \,rdr 
			+ \int_{\delta(1-R)}^{1-R} \frac{1}{r} \,dr 
			\right) RdR \\
		&= 4 c^2 \left(\frac{1}{2} + \ln \delta^{-1}\right) \hat{R}^2,
	\end{align*}
	and hence
	\begin{align*}
		H
		&\ge (-\Delta_{\Omega^2}^\mathcal{N})P 
			+ (-\Delta_{\Omega^2}^\mathcal{N})Q 
			+ (1-\mu)PfP + (1-\mu^{-1})QfQ \\
		&\ge (1-\mu) 2 c^2 (1 + 2\ln \delta^{-1}) \hat{R}^2 \, P
		+ \left( \xi^2 - (\mu^{-1}-1) \frac{c^2}{\delta^2 (1-\hat{R})^2} \right)Q.
	\end{align*}
	With the parameter $\kappa$ from \eqref{abs-rel-identity-sum}
	reintroduced into \eqref{minimize-energy} and \eqref{Neumann-operator},
	we bound
	\begin{align*}
		H &\ge 2 c^2 (1-\mu) (1 + 2\ln \delta^{-1}) \hat{R}^2 \kappa \, P
		+ \left( \xi^2(1-\kappa) - \frac{(\mu^{-1}-1)\kappa}{\delta^2 (1-\hat{R})^2} \right)Q \\
		&\ge 0.1077 c^2,
	\end{align*}
	where the lower bound was found by numerical optimization,
	with $\mu=0.8669$, $\delta=0.5556$, $\hat{R}=0.6513$, $\kappa=0.4387$,
	$\xi^2 \ge 3.389$.
	Summing up, we have
	\begin{multline} \label{final-local-bound}
		\sum_j \int_{\Omega^n} |D_j u|^2 dx \\
		\ge \frac{1}{n} \sum_{j<k} \int_{\Omega^{n-2}} \int_{\Omega^2}
			\left( (1-\kappa)(|D_ju|^2 + |D_ku|^2) 
				+ \kappa f(\bx_j,\bx_k) \right) d\bx_j d\bx_k \, dx' \\
		\ge (n-1) \frac{c_\Omega c^2}{|\Omega|} \int_{\Omega^n} |u|^2 dx,
	\end{multline}
	with $c_\Omega \ge 0.1077 \cdot \pi/2 \ge 0.169$.
	
	In the case that $\Omega = (-1,1)^2 \supseteq B_1(0)$,
	we can use the same $f$ and $g$ as above
	(extended by zero outside $B_1(0)^2$),
	so that the bound on $\|QfQ\|$ is unchanged, 
	$\|PfP\|$ is multiplied by the square of the area ratio $\frac{\pi}{4}$,
	and $(-\Delta_{\Omega^2}^\mathcal{N})Q \ge \frac{\pi^2}{4} Q$.
	We find
	\begin{align*}
		H &\ge \frac{\pi^2 c^2}{8} (1-\mu) (1 + 2\ln \delta^{-1}) \hat{R}^2 \kappa \, P
		+ \left( \frac{\pi^2}{4}(1-\kappa) - \frac{(\mu^{-1}-1)\kappa}{\delta^2 (1-\hat{R})^2} \right)Q \\
		&\ge 0.056 c^2,
	\end{align*}
	with $\mu=0.8681$, $\delta=0.5928$, $\hat{R}=0.661$, $\kappa=0.387$,
	and hence \eqref{final-local-bound} holds 
	with the constant $c_\Omega \ge 0.056 \cdot 4/2 = 0.112$.
\end{proof}

\begin{rem}
	More general convex domains can be treated in a similar way using
	\cite{Payne-Weinberger:60}. 
	Also note that the bound \eqref{energy_in_ball} 
	for the disk holds with
	$c_\Omega = \pi \xi^2 \approx 10.65$ in the case that $\alpha = 1$ 
	(compare with Lemma 5 in \cite{Dyson-Lenard:67}).
	For possible ways of improving the constant $c_\Omega$ in the general case, 
	see Appendix A. 
\end{rem}

\section{A Lieb-Thirring inequality for anyons}
	\label{sec:Lieb-Thirring}

	Given a \emph{normalized} $N$-anyon wave function 
	$u \in \Hc_{\alpha_0}^N$ 
	we will in the following denote
	by
	$$
		\rho(\bx) := \sum_{j=1}^N \int_{\R^{2(N-1)}} 
			|u(\bx_1,\ldots,\bx_{j-1},\bx,\bx_{j+1},\ldots,\bx_N)|^2
			\prod_{k \neq j} d\bx_k
	$$
	the one-particle density,
	such that $\int_{\R^2} \rho(\bx) \,d\bx = N$.
	We start by reformulating 
	Lemma \ref{lem:energy_in_ball} in terms of $\rho$.

\begin{lem}[Local exclusion principle] \label{lem:Pauli_on_box}
	Let $u \in \Dc_{0,\alpha}^N$  be an $N$-anyon wave function on $\R^2$
	and $\Omega \subseteq \R^2$ a simply connected domain on which 
	\eqref{energy_in_ball} holds for some constant $c_\Omega$.
	Then
	\begin{equation} \label{Pauli_on_box}
		T_\Omega := \sum_{j=1}^N \int_{\R^{2N}} |D_j u|^2 \chi_\Omega(\bx_j) \,dx
		\ \ge \ \frac{c_\Omega C_{\alpha,N}^2}{|\Omega|} 
		\left( \int_\Omega \rho \ - 1 \right).
	\end{equation}
\end{lem}
\begin{proof}
	Using that
	$$
		1 = \prod_{k=1}^N \big( \chi_\Omega(\bx_k) + (1 - \chi_\Omega(\bx_k)) \big)
		= 	\sum_{A \subseteq \{1,\ldots,N\}} 
			\prod_{k \in A} \chi_\Omega(\bx_k) 
			\prod_{k \notin A} (1 - \chi_\Omega(\bx_k)),
	$$
	the l.h.s. of \eqref{Pauli_on_box} is
	\begin{multline*}
		T_\Omega = \int_{\R^{2N}}
			\sum_{j=1}^N |D_j u|^2 \, \chi_\Omega(\bx_j) 
			\sum_{A \subseteq \{1,\ldots,N\}} 
			\prod_{k \in A} \chi_\Omega(\bx_k) 
			\prod_{k \notin A} (1 - \chi_\Omega(\bx_k)) \,dx \\
		= \sum_{A \subseteq \{1,\ldots,N\}}
			\int_{\R^{2N}} \sum_{j \in A} |D_j u|^2 
			\prod_{k \in A} \chi_\Omega(\bx_k) 
			\prod_{k \notin A} (1 - \chi_\Omega(\bx_k)) \,dx.
	\end{multline*}
	We now apply Lemma \ref{lem:energy_in_ball} to each term 
	in the first summation above, which involves
	a partition $A$ of the $N$ particles into 
	$n := |A|$ of them being inside the domain $\Omega$, 
	while the remaining $N-n$ residing outside, 
	and therefore whose contributions to the magnetic potentials 
	$\bA_{j \in A}$
	can be gauged away. 
	Thus, we find
	\begin{multline*}
		T_\Omega \ge 
			\sum_{A \subseteq \{1,\ldots,N\}}
			\frac{c_\Omega C_{\alpha,|A|}^2}{|\Omega|} \left( |A| - 1 \right)_+ \!\!\!\!
			\int\limits_{\bx_{k \notin A} \notin \Omega} \ 
			\int\limits_{\bx_{k \in A} \in \Omega}
			|u|^2 \prod_{k \in A} d\bx_k \prod_{k \notin A} d\bx_k \\
		\ge \frac{c_\Omega C_{\alpha,N}^2}{|\Omega|}
			\int_{\R^{2N}} |u|^2 \!\!\!
			\sum_{A \subseteq \{1,\ldots,N\}} (|A| - 1)
			\prod_{k \in A} \chi_\Omega(\bx_k) 
			\prod_{k \notin A} (1 - \chi_\Omega(\bx_k)) \,dx.
	\end{multline*}
	We then revert the above procedure using 
	$\int_{\R^{2N}} |u|^2 = 1$ and
	\begin{multline*}
		\sum_{A \subseteq \{1,\ldots,N\}} \underbrace{|A|}_{=\sum_{j \in A}}
			\prod_{k \in A} \chi_\Omega(\bx_k) 
			\prod_{k \notin A} (1 - \chi_\Omega(\bx_k)) \\
		= \sum_{j=1}^N \sum_{A \subseteq \{1,\ldots,N\}} \chi_\Omega(\bx_j)
			\prod_{k \in A} \chi_\Omega(\bx_k) 
			\prod_{k \notin A} (1 - \chi_\Omega(\bx_k))
		= \sum_{j=1}^N \chi_\Omega(\bx_j),
	\end{multline*}
	which produces \eqref{Pauli_on_box}.
\end{proof}

\begin{lem}[Local uncertainty principle] 
	\label{lem:box-LT-anyons}
	Let $u \in \Dc_{0,\alpha}^N$  be an $N$-anyon wave function on $\R^2$,
	and $Q$ a square with area $|Q|$.
	Then
	\begin{equation} \label{box-LT-anyons}
		\sum_{j=1}^N \int_{\R^{2N}} |D_j u|^2 \chi_Q(\bx_j) \,dx
		\ \ge \ C_2'\, ({\textstyle \int_Q \rho})^{-1} \int_Q \left[
			\rho(\bx)^{\frac{1}{2}} - \left( \frac{\int_Q \rho}{|Q|} \right)^{\frac{1}{2}}
			\right]_+^4 d\bx.
	\end{equation}
\end{lem}
\begin{proof}
	We use $|D_j u| \ge \big| \nabla_j |u| \big|$ 
	and then apply the Neumann Lieb-Thirring inequality 
	given in Theorem \ref{thm:many-particle-NLT} 
	in the appendix to the bosonic wave function $|u|$.
\end{proof}

\begin{rem}
	A sufficient bound can also be obtained 
	by means of Poincar\'e and Sobolev
	inequalities applied to
	$\rho^{\frac{1}{2}}$
	(we thank R. Seiringer for this remark).
	Note that the limit case $Q = \R^2$, $N=1$, 
	is just a special case of the Sobolev-type inequality
	$\|\nabla u\|_2^2 \ge C_p \|u\|_2^{-4/(p-2)} \|u\|_p^{2p/(p-2)}$
	with $p=4$.
\end{rem}

\begin{lem} \label{lem:Neumann-integral}
	For any domain $\Omega \subseteq \R^2$ with 
	finite area $|\Omega|$
	we have
	\begin{equation} \label{Neumann-integral-lemma}
		\int_\Omega \left[ \rho^{\frac{1}{2}} 
			- \left( \frac{\int_\Omega \rho}{|\Omega|} \right)^{\frac{1}{2}} 
			\right]_+^4
		\ge (1-4\epsilon) \int_\Omega \rho^2 
			+ (2 - \epsilon^{-1}) \frac{(\int_\Omega \rho)^2}{|\Omega|},
	\end{equation}
	for arbitrary $\epsilon > 0$.
\end{lem}
\begin{proof}
	First note that
	$$
		\int_\Omega \left[ \rho^{\frac{1}{2}} 
			- \left( \frac{\int_\Omega \rho}{|\Omega|} \right)^{\frac{1}{2}} 
			\right]_+^4
		\ge \int_\Omega \left[ \rho^{\frac{1}{2}}
			- \left( \frac{\int_\Omega \rho}{|\Omega|} \right)^{\frac{1}{2}} 
			\right]^4 \ \ 
			- \frac{(\int_\Omega \rho)^2}{|\Omega|},
	$$
	where the first integral on the r.h.s. is equal to
	$$
			\int_\Omega \rho^2 
			- 4\int_\Omega \rho^{\frac{1}{2}} \left( \frac{\int_\Omega \rho}{|\Omega|} \right)^{\frac{3}{2}}
			- 4\int_\Omega \rho^{\frac{3}{2}} \left( \frac{\int_\Omega \rho}{|\Omega|} \right)^{\frac{1}{2}}
			+ 7 \frac{(\int_\Omega \rho)^2}{|\Omega|}.
	$$
	Together with H\"older's inequality applied to the negative terms,
	$$
		\int_\Omega \rho^{\frac{1}{2}} 
			\le \left( \int_\Omega \rho \right)^{\frac{1}{2}} |\Omega|^{\frac{1}{2}},
		\quad \textrm{and} \quad
		\int_\Omega \rho \cdot \rho^{\frac{1}{2}} \left( \frac{\int_\Omega \rho}{|\Omega|} \right)^{\frac{1}{2}} 
			\le \epsilon \int_\Omega \rho^2 + \frac{1}{4\epsilon} \frac{(\int_\Omega \rho)^2}{|\Omega|},
	$$
	this
	produces the inequality \eqref{Neumann-integral-lemma}.
\end{proof}

\begin{thm}[Kinetic energy inequality for anyons] \label{thm:kinetic-LT-anyons}
	Let $u \in \Dc_{0,\alpha}^N$  be an $N$-anyon wave function 
	on $\R^2$, with $N \ge 2$.
	Then
	\begin{equation} \label{kinetic-LT-anyons}
		\sum_{j=1}^N \int_{\R^{2N}} |D_j u|^2 \,dx
		\ \ge \ C_{\textup{K}} C_{\alpha,N}^2 \int_{\R^2} \rho(\bx)^2 \,d\bx,
	\end{equation}
	for some positive constant $C_{\textup{K}}$.
\end{thm}
\begin{proof}
	Given $u$, and hence $\rho$, smooth and
	supported on some square $Q_0 \subseteq \R^2$
	(it is clearly enough to consider this case),
	we split up the domain $Q := Q_0$ into ever smaller subsquares 
	according to the following algorithm 
	(see Figure \ref{fig:splitting}):
	\begin{figure}[t]
		\centering
		\psfrag{A}{A}
		\psfrag{B}{B}
		\psfrag{T_Q}{$Q_0$}
		\includegraphics{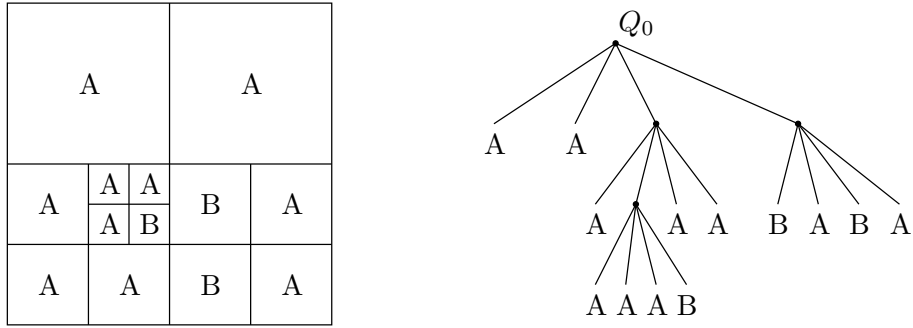}
		\caption{Example of a splitting of $Q_0$ and a corresponding
			tree $\mathbb{T}$ of subsquares. For the B-square at
			level 3 in the tree, the set $\mathcal{A}(Q)$ of associated 
			A-squares consists of 8 elements, while for the two
			B-squares at level 2, $\mathcal{A}(Q)$ coincide and
			has 4 elements.}
		\label{fig:splitting}
	\end{figure}
	\begin{itemize}
		\item A given square $Q$ is split into four subsquares
			$Q'_{j \in \{1,2,3,4\}}$ s.t. $|Q'_j| = |Q|/4$.
		
		\item Whenever $\int_{Q'_j} \rho < 2$, we will not split that 
			square $Q'_j$ further, and we mark it A.
			
		\item If all four squares $Q'_j$ are marked A, then
			we back up to the bigger square $Q$ and mark it B.
			(Thus, we then stop splitting $Q$ 
			and the subsquares $Q'_j$ are discarded.)
		
		\item For each of the unmarked squares $Q'_j$ 
			(i.e. those for which $\int_{Q'_j} \rho \ge 2$) 
			we iterate the splitting algorithm with $Q := Q'_j$.
	\end{itemize}
	In the end 
	(note that the procedure will eventually stop
	since $\rho$ is integrable),
	we will have a system of subsquares $Q$ 
	organized in a tree $\mathbb{T}$, such that:
	\begin{itemize}
		\item The whole domain $Q_0$ is covered by squares of type A or B.
		
		\item Let us denote, for a given B-square $Q_B$,
		\begin{equation*} \label{A-ancestors}
			\mathcal{A}(Q_B) := \left\{ \begin{array}{l}
				\text{All A-squares $Q_A \in \mathbb{T}$ that can be found by going }\\
				\text{back in the tree from $Q_B$ (possibly all the way to }\\
				\text{$Q_0$) and then one step forward.} 
			\end{array} \right\},
		\end{equation*}
		then every A-square $Q_A \in \mathbb{T}$ satisfies
		$Q_A \in \mathcal{A}(Q_B)$ 
		for some B-square $Q_B \in \mathbb{T}$.
		(Note that at least one B-square can be found among the leaves
		in the highest level of every branch of the tree.)
		
		\item $0 \le \int_{Q_A} \rho < 2$ for every A-square $Q_A$.
		
		\item $2 \le \int_{Q_B} \rho < 8$ for every B-square $Q_B$.
	\end{itemize}
	Finally, we also divide the A-squares into a subclass A$_1$ 
	for those subsquares on which $\rho$ is approximately constant, 
	and A$_2$ for those subsquares with a non-constant density:
	$$
		\mathcal{A}_1 := \left\{ \textstyle \text{all A-squares $Q \in \mathbb{T}$ s.t. $\int_Q \rho^2 \le c \frac{(\int_Q \rho)^2}{|Q|}$} \right\},
	$$
	$$
		\mathcal{A}_2 := \left\{ \textstyle \text{all A-squares $Q \in \mathbb{T}$ s.t. $\int_Q \rho^2 > c \frac{(\int_Q \rho)^2}{|Q|}$} \right\},
	$$
	for some fixed constant $c>1$, to be chosen below.
	We will then split the kinetic energy integral $T$ in the l.h.s. of 
	\eqref{kinetic-LT-anyons} into a sum over all the marked subsquares 
	$Q \in \mathcal{A}_1 \cup \mathcal{A}_2 \cup \mathcal{B}$ 
	forming the leaves of the tree $\mathbb{T}$, 
	and consider the three different types of squares separately. 
	
	Let us first consider the kinetic energy $T_{Q_B}$ on any B-square $Q = Q_B$.
	We split it up into two parts using $\kappa \in (0,1)$ and 
	apply both Lemma \ref{lem:box-LT-anyons} (local uncertainty principle) 
	and Lemma \ref{lem:Pauli_on_box} (local Pauli principle)
	to conclude
	\begin{multline*}
		T_{Q_B} 
		\ge \kappa \frac{C_2'}{\int_Q \rho} \int_Q \left[
				\rho^{\frac{1}{2}} - \left( \frac{\int_Q \rho}{|Q|} \right)^{\frac{1}{2}}
				\right]_+^4
			+ (1-\kappa) \frac{c_Q C_{\alpha,N}^2}{|Q|} 
				\left( \int_Q \rho - 1 \right) \\
		\ge \frac{\kappa C_2'}{8}(1-4\epsilon) \int_Q \rho^2 
			+ \left( \frac{\kappa C_2'}{8}(2 - \epsilon^{-1}) 
			+ (1-\kappa)\frac{7c_Q C_{\alpha,N}^2}{64} \right) 
			\frac{(\int_Q \rho)^2}{|Q|},
	\end{multline*}
	where we also used that $x - 1 \ge \frac{7}{64}x^2$ for $2 \le x < 8$,
	as well as Lemma \ref{lem:Neumann-integral}.
	Choosing $\kappa$ and $\epsilon$ appropriately, we conclude
	\begin{equation} \label{energy_on_B-square}
		T_{Q_B} \ge C_{\alpha,N}^2 \left( c_1 \int_{Q_B} \rho^2 
			+ c_2 \frac{(\int_{Q_B} \rho)^2}{|Q_B|} \right),
	\end{equation}
	with $c_1,c_2 > 0$.
	
	For the A$_2$-squares $Q=Q_A$, we use that the energy given by 
	the local uncertainty principle
	is large due to a
	sufficiently non-constant $\rho$:
	\begin{equation*}
		T_{Q_A} 
		\ge \frac{C_2'}{\int_Q \rho} \int_Q \left[
				\rho^{\frac{1}{2}} - \left( \frac{\int_Q \rho}{|Q|} \right)^{\frac{1}{2}}
				\right]_+^4 
		\ge \frac{C_2'}{2}\left( 1-4\epsilon - (\epsilon^{-1} - 2)c^{-1} \right) \int_Q \rho^2,
	\end{equation*}
	for $0 < \epsilon < 1/4$,
	where we again used 
	Lemmas \ref{lem:box-LT-anyons} and \ref{lem:Neumann-integral}.
	Taking $\epsilon := 1/8$ and $c := 24$ we find
	\begin{equation} \label{energy_on_A2-square}
		T_{Q_A} \ge \frac{C_2'}{8} \int_{Q_A} \rho^2.
	\end{equation}
	
	Lastly, we show that the remaining squares, of type A$_1$, have
	a negligible contribution to the energy compared to that
	already obtained from the Pauli energy on the B-squares.
	Namely, as observed above, every subsquare of type A$_1$ is
	contained in $\mathcal{A}_1(Q_B) := \mathcal{A}(Q_B) \cap \mathcal{A}_1$ 
	for some B-square $Q_B$.
	Now, note that for each B-square $Q_B \in \mathbb{T}$, 
	say at a level $k \in \N$ in the tree, 
	we have
	$$
		\frac{(\int_{Q_B} \rho)^2}{|Q_B|} \ge \frac{4}{4^{-k}|Q_0|},
	$$
	while the total integral of $\rho^2$ over 
	all A$_1$-squares associated with $Q_B$ is at most 
	$$
		\sum_{Q \in \mathcal{A}_1(Q_B)} \int_{Q} \rho^2
		\le \sum_{j=1}^k \  
			\sum_{\substack{Q \in \mathcal{A}_1(Q_B) \\ \text{at level $j$} }}
			c \frac{(\int_Q \rho)^2}{|Q|}
		\le \sum_{j=1}^k 3c \frac{4}{4^{-j}|Q_0|}
		\le 96 \frac{4^{k+1}}{|Q_0|}.
	$$
	Hence, by \eqref{energy_on_B-square},
	\begin{equation} \label{energy_on_A1-squares}
		T_{Q_B} \ge C_{\alpha,N}^2 \left( 
			c_1 \int_{Q_B} \rho^2 
			+ \frac{c_2}{96} \sum_{Q \in \mathcal{A}_1(Q_B)} \int_{Q} \rho^2
			\right).
	\end{equation}
	
	Summing everything up, it follows from 
	\eqref{energy_on_A2-square} and \eqref{energy_on_A1-squares} that
	the total kinetic energy is
	\begin{multline*}
		T = \sum_{j=1}^N \int_{\R^{2N}} |D_j u|^2 
			\left( \sum_{Q_A \in \mathcal{A}} \chi_{Q_A}(\bx_j) 
			+ \sum_{Q_B \in \mathcal{B}} \chi_{Q_B}(\bx_j) \right) dx \\
		\ge \sum_{Q_A \in \mathcal{A}_2} T_{Q_A} + \sum_{Q_B \in \mathcal{B}} T_{Q_B}
		\ge C_{\textup{K}} C_{\alpha,N}^2 \int_{Q_0} \rho^2,
	\end{multline*}
	for a positive constant 
	$C_{\textup{K}} := \min\{c_1,c_2/96,C_2'/8\}$.
\end{proof}

\begin{cor}[Lieb-Thirring inequality for anyons]
	Let $u \in \Dc_{0,\alpha}^N$  be an $N$-anyon wave function 
	and $V$ a real-valued potential on $\R^2$. Then
	\begin{equation} \label{LT-anyons}
		\sum_{j=1}^N \int_{\R^{2N}} \left( |D_j u|^2 + V(\bx_j)|u|^2 \right) \,dx
		\ \ge \ -C_{\textup{LT}} C_{\alpha,N}^{-2} \int_{\R^2} |V_-(\bx)|^2 \,d\bx,
	\end{equation}
	for some positive constant 
	$C_{\textup{LT}} = (4C_{\textup{K}})^{-1}$. 
\end{cor}
\begin{proof}
	The l.h.s. is bounded below by
	$$
		C_{\textup{K}} C_{\alpha,N}^2 \int_{\R^2} \rho^2 \,d\bx - \int_{\R^2} |V_-|\rho \,d\bx
		\ge -\frac{1}{4} C_{\textup{K}}^{-1} C_{\alpha,N}^{-2} \int_{\R^2} V_-^2 \,d\bx,
	$$
	where we used 
	$\int |V_-|\rho \le (\int V_-^2 )^{\frac{1}{2}} (\int \rho^2 )^{\frac{1}{2}}$
	and minimized w.r.t. $\int \rho^2$.
\end{proof}

\begin{rem}
	Theorem \ref{thm:kinetic-LT-anyons} immediately implies the rough bound
	\begin{equation} \label{gs-energy-bound}
		T \ge C_{\textup{K}} C_{\alpha,N}^2 \int_\Omega \rho^2 \,d\bx 
		\ge \frac{C_{\textup{K}} C_{\alpha,N}^2}{|\Omega|} \left( \int_\Omega \rho \cdot 1 \,d\bx \right)^2
		= C_{\textup{K}} C_{\alpha,N}^2 \frac{N^2}{|\Omega|}
	\end{equation}
	for the ground state energy of a non-interacting gas of anyons 
	supported on a domain $\Omega \subseteq \R^2$.
\end{rem}

\section*{Appendix A: Improvements of the local energy}

	Here we give some comments on how the constant in the explicit bound
	for the energy in Lemma \ref{lem:energy_in_ball} could
	be improved.
	One alternative approach for taking the full many-particle domain 
	$\Omega^N$ into account is to extend 
	the Hardy inequality \eqref{many-anyon-Hardy} using a variant of 
	Theorem 2.2 in \cite{Balinsky:03}, which here again has been 
	modified to account for the underlying 
	symmetry.

\begin{thm}[Many-anyon Hardy on a disk] \label{thm:many-anyon-Hardy-ball}
	Let $\Omega := B_\lambda(0)$ be a disk of radius $\lambda$ in $\R^2$ and
	let $u \in \Dc^N_{0,\alpha}$. 
	Then
	\begin{equation} \label{many-anyon-Hardy-ball}
		\int_{\Omega^N} \sum_{j=1}^N |D_j u|^2 \,dx
		\ge \frac{C_{\alpha,N}^2}{N\lambda^2} 
			\int_{\Omega^N} \sum_{i<j}
			f\left( \frac{\bx_i + \bx_j}{2\lambda}, \frac{\bx_i - \bx_j}{2\lambda\sqrt{1-R^2}} \right) 
			|u|^2 \,dx,
	\end{equation}
	where
	$$
		f(\bR, \tilde{\br}) := \frac{16}{\gamma_R^2 (1-R^2)}
		\left| (1+z)^{1+\frac{1}{\gamma_R}} (1-z)^{1-\frac{1}{\gamma_R}}
			-  (1-z)^{1+\frac{1}{\gamma_R}} (1+z)^{1-\frac{1}{\gamma_R}} \right|^{-2},
	$$
	$z$ is the complexification of the renormalized coordinate $\tilde{\br}$ 
	with imaginary axis given by the unit vector $\bR/R$,
	and $\gamma_R := 1 - \frac{2}{\pi}\arcsin(R)$.
\end{thm}
\begin{proof}
	\begin{figure}[t]
		\centering
		\psfrag{F}{$F$}
		\psfrag{0}{$0$}
		\psfrag{T_R}{$\bR$}
		\psfrag{T_OR}{$\Omega_{\bR}$}
		\psfrag{T_dR}{$\delta(\bR)$}
		\psfrag{T_gp}{$\gamma\pi$}
		\includegraphics{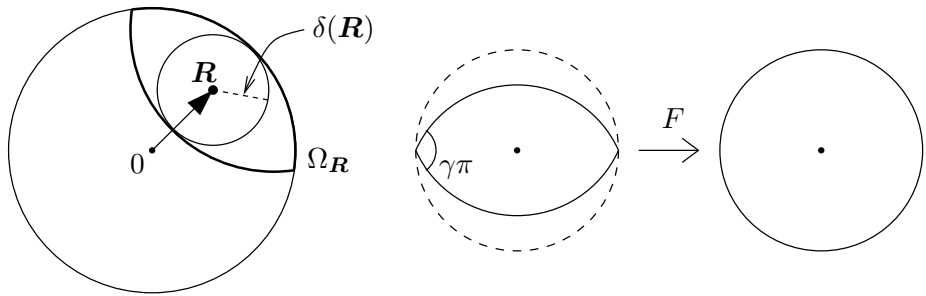}
		\caption{The set $\Omega_{\bR}$, which can be mapped 
			conformally to the unit disk by the map $F$.}
		\label{fig:eye}
	\end{figure}
	We proceed as in the proof of Theorem \ref{thm:many-anyon-Hardy},
	but this time parameterize not just $\Omega \circ \Omega$, but
	the whole set $\Omega^2$ by
	the relative coordinates $\bR$, $\br$.
	For $\bR$ at a distance $\lambda R$, $0 \le R \le 1$,
	from the center $0$ of the disk,
	the total range of allowed $\br$ forms an eye-shaped set $\Omega_{\bR}$ 
	given by the intersection of two disks (see Figure \ref{fig:eye}).
	After rescaling by $(1-R^2)^{-\frac{1}{2}}\lambda^{-1}$ 
	and choosing axes so that the corners of the eye are at the
	points $z=-1$ resp. $z=1$ in the complex plane, 
	we can map the resulting set 
	conformally onto the unit disk by the sequence of maps
	$$
		F: \quad
		z \mapsto \frac{1+z}{1-z} =: w,
		\quad
		w \mapsto w^{\frac{1}{\gamma}} =: \zeta,
		\quad
		\zeta \mapsto \frac{\zeta-1}{\zeta+1} =: F(z),
	$$
	where $\gamma\pi = \gamma_R \pi$ is the angle at the corner of the eye.
	The resulting map is hence
	\begin{equation} \label{conformal-map}
		F(z) = \frac{(1+z)^{\frac{1}{\gamma}} - (1-z)^{\frac{1}{\gamma}}}
					{(1+z)^{\frac{1}{\gamma}} + (1-z)^{\frac{1}{\gamma}}}
			=: \xi + i\eta,
	\end{equation}
	which is \emph{antisymmetric} w.r.t. the antipodal map $z \mapsto -z$.
	Now, under the coordinate transformation $F$, the integral w.r.t. $\br$ in 
	\eqref{many-anyon-integral-split} becomes
	\begin{multline*}
		\int_{\Omega_{\bR}} |(D_j - D_k)u(\br)|^2 \,d\br 
		= \int_{B_1(0)} |(D_{F(\tilde{\br})}v(\xi,\eta)|^2 \,d\xi d\eta \\
		\ge C_{\alpha,N}^2 \int_{B_1(0)} \frac{|v|^2}{\xi^2 + \eta^2} \,d\xi d\eta
		= \frac{C_{\alpha,N}^2}{(1-R^2)\lambda^2} 
			\int_{\Omega_{\bR}} \frac{|F'(z)|^2}{|F(z)|^2} |u|^2 \,d\br,
	\end{multline*}
	where we again applied an annular decomposition of $B_1(0)$
	and Lemma \ref{lem:magnetic-Hardy},
	and used the antipodal symmetry of 
	$v(\xi,\eta) := u(F^{-1}(\xi+i\eta))$.
	Finally, we have by the definition of $f$ that 
	$|F'(z)|^2/|F(z)|^2 = (1-R^2) f(\bR,\tilde{\br})$.
\end{proof}

	Using the above Hardy potential $f$ in \eqref{Neumann-operator},
	together with more precise estimates of the lowest eigenvalue
	of the corresponding operator $H$,
	and the restriction that the eigenfunction should be antipodal-symmetric,
	would almost certainly produce a significantly 
	better bound on the constant $c_\Omega$
	and hence on the energy of 
	a gas of anyons.

\section*{Appendix B: A local Lieb-Thirring inequality with Neumann boundary conditions}

	In this appendix we derive certain 
	bosonic and fermionic kinetic energy inequalities
	on domains with Neumann boundary conditions.
	These 
	follow straightforwardly from a recent method
	due to Rumin \cite{Rumin:10}.

\begin{thm}
	\label{thm:single-particle-NLT}
	Let $Q$ be a cube in $\R^d$ with volume $|Q|$.
	Then
	\begin{equation} \label{box-LT}
		\sum_{j=1}^N \|\nabla \phi_j\|^2_{L^2(Q)} 
		\ \ge \ C_d' \, ({\textstyle \int_Q \rho})^{-\frac{2}{d}} \int_Q \left[
			\rho(\bx)^{\frac{1}{2}} - \left( \frac{\int_Q \rho}{|Q|} \right)^{\frac{1}{2}}
			\right]_+^{\frac{2(d+2)}{d}} d\bx,
	\end{equation}
	where $\phi_j \in H^1(Q)$ 
	and $\rho(\bx) := \sum_{j=1}^N |\phi_j(\bx)|^2$.
	
	Moreover, if $\{\phi_j\}$ are orthonormal in $L^2(Q)$, then
	\begin{equation} \label{box-LT-ON}
		\sum_{j=1}^N \|\nabla \phi_j\|^2_{L^2(Q)} 
		\ \ge \ C_d' \int_Q \left[
			\rho(\bx)^{\frac{1}{2}} - |Q|^{-\frac{1}{2}}
			\right]_+^{\frac{2(d+2)}{d}} d\bx.
	\end{equation}
\end{thm}

\begin{proof}
	For any $e \ge 0$ and $\phi \in L^2(Q)$, let
	$$
		\phi = \phi^{e,+} + \phi^{e,-},
		\quad
		\phi^{e,+} := P_{\{-\Delta^{\mathcal{N}}_Q \ge e\}} \phi,
		\quad
		\phi^{e,-} := P_{\{-\Delta^{\mathcal{N}}_Q < e\}} \phi,
	$$
	and note that for $\phi \in H^1(Q)$
	(interpreted in terms of quadratic forms) 
	$$
		\int_0^\infty \| \phi^{e,+} \|_{L^2(Q)}^2 \,de
		= \langle \phi, \int_0^\infty \underbrace{P_{\{-\Delta^{\mathcal{N}}_Q \ge e\}}}_{=\int_{\R} 1_{\{\lambda \ge e\}} dP(\lambda) } de \,\phi \rangle
		= \langle \phi, \underbrace{-\Delta^{\mathcal{N}}_Q}_{=\int_{\R} \lambda \,dP(\lambda)} \phi \rangle.
	$$
	Denote the eigenvalues and orthonormal eigenfunctions 
	of $-\Delta^{\mathcal{N}}_Q$ 
	by $\{\lambda_k\}_{k=0}^\infty$ resp. $\{u_k\}_{k=0}^\infty$.
	Then for each $\bx \in Q$,
	\begin{multline} \label{LT-eigenvalue-estimate}
		|\phi^{e,-}(\bx)|^2 
		= \left| \left(P_{\{-\Delta^{\mathcal{N}}_Q < e\}} \phi\right)(\bx) \right|^2
		= \left| \sum_{\lambda_k < e} \langle u_k, \phi \rangle u_k(\bx) \right|^2 \\
		= \left| \left\langle \sum_{\lambda_k < e} \overline{u_k(\bx)} u_k, \phi \right\rangle \right|^2
		\le \left( \sum_{\lambda_k < e} |u_k(\bx)|^2 \right) \|\phi\|^2 \\
		\le \left( \frac{1}{|Q|} + \sum_{0 < \lambda_k < e} \frac{2^d}{|Q|} \right) \|\phi\|^2,
	\end{multline}
	and hence $|\phi^{e,-}(\bx)|^2 \le (|Q|^{-1} + C_d e^{\frac{d}{2}}) \|\phi\|^2$
	by the well-known asymptotics for the Neumann eigenvalues on $Q$.
	
	Now, by the triangle inequality in $\C^N$ we have for arbitrary 
	$\{\phi_j\}_{j=1}^N \subseteq L^2(Q)$
	$$
		\left( \sum_{j=1}^N |\phi_j^{e,+}(\bx)|^2 \right)^{\frac{1}{2}}
		\ge \left[ 
			\left( \sum_{j=1}^N |\phi_j(\bx)|^2 \right)^{\frac{1}{2}} -
			\left( \sum_{j=1}^N |\phi_j^{e,-}(\bx)|^2 \right)^{\frac{1}{2}}
		\right]_+ \!,
		\quad \bx \in Q,
	$$
	and hence we find that for $\rho(\bx) := \sum_j |\phi_j(\bx)|^2$
	and $\{\phi_j\}_{j=1}^N \subseteq H^1(Q)$ 
	\begin{multline} \label{LT-integral}
		\sum_{j=1}^N \|\nabla \phi_j\|^2 
		= \sum_j \int_Q \int_0^\infty |\phi_j^{e,+}(\bx)|^2 \,de \,d\bx \\
		\ge \int_Q \int_0^\infty \left[ 
			\rho(\bx)^{\frac{1}{2}} - \left( (|Q|^{-1} 
				+ C_d e^{\frac{d}{2}}) \int_Q \rho(\by) \,d\by \right)^{\frac{1}{2}}
			\right]_+^2 de \,d\bx \\
		\ge C_d \int_Q \rho(\by) \,d\by \int_Q \int_0^\infty \left[ 
			\left( \frac{\rho(\bx)}{C_d \int_Q \rho} \right)^{\frac{1}{2}}
			- \left( \frac{1}{C_d |Q|} \right)^{\frac{1}{2}} - e^{\frac{d}{4}}
			\right]_+^2 de \,d\bx \\
		= \frac{d^2 C_d \int_Q \rho}{(d+2)(d+4)} \int_Q \left[ 
			\left( \frac{\rho(\bx)}{C_d \int_Q \rho} \right)^{\frac{1}{2}}
			- \left( \frac{1}{C_d |Q|} \right)^{\frac{1}{2}}
			\right]_+^{\frac{2(d+2)}{d}} d\bx \\
		= C_d'\, ({\textstyle \int_Q \rho})^{1-\frac{d+2}{d}} \int_Q \left[
			\rho(\bx)^{\frac{1}{2}} - \left( \frac{\int_Q \rho}{|Q|} \right)^{\frac{1}{2}}
			\right]_+^{\frac{2(d+2)}{d}} d\bx,
	\end{multline}
	with $C_d' := d^2 C_d^{-\frac{2}{d}} \big/ (d+2)(d+4)$.
	
	In the case that $\{\phi_j\}$ are orthonormal, 
	Bessel's inequality applies in \eqref{LT-eigenvalue-estimate}:
	$$
		\sum_j \left| \left\langle \sum_{\lambda_k < e} \overline{u_k(\bx)} u_k, \phi_j \right\rangle \right|^2
		\le \left\| \sum_{\lambda_k < e} \overline{u_k(\bx)} u_k \right\|^2
		= \sum_{\lambda_k < e} |u_k(\bx)|^2,
	$$
	which then replaces every occurance of 
	$\int_Q \rho$ in \eqref{LT-integral} by $1$.
\end{proof}

\begin{rem}[Generalization]
	If we know how 
	$\Xi_\Omega(e,\bx) := \sum_{\lambda_k < e} |u_k(\bx)|^2$
	behaves for a given domain $\Omega \subseteq \R^d$, 
	then we can evaluate
	\begin{equation} \label{generalized-LT}
		\sum_{j=1}^N \langle \phi_j, -\Delta_\Omega \phi_j \rangle
		\ge \int_\Omega \int_0^\infty \left[ 
			\rho(\bx)^{\frac{1}{2}} - \left( \Xi_\Omega(e,\bx) \int_\Omega \rho \right)^{\frac{1}{2}}
			\right]_+^2 de \,d\bx \\
	\end{equation}
	for either the Neumann or Dirichlet Laplacian on $\Omega$.
	Again, $\int_\Omega \rho$ in the r.h.s. is replaced by $1$ 
	in the case that $\{\phi_j\}$ are orthonormal.
\end{rem}

\begin{thm}[Many-particle version] \label{thm:many-particle-NLT}
	Let $Q$ be a cube in $\R^d$ with volume $|Q|$,
	and let $u \in H^1(\R^{dN})$ be an $N$-particle wave function.
	Then
	\begin{equation} \label{box-LT-many}
		\sum_{j=1}^N \int_{\R^{dN}} |\nabla_j u|^2 \chi_Q(\bx_j) \,dx
		\ \ge \ C_d' ({\textstyle \int_Q \rho})^{-\frac{2}{d}} \int_Q \left[
			\rho(\bx)^{\frac{1}{2}} - \left( \frac{\int_Q \rho}{|Q|} \right)^{\frac{1}{2}}
			\right]_+^{\frac{2(d+2)}{d}} \! d\bx,
	\end{equation}
	where $\rho(\bx) := \sum_{j=1}^N \int_{\R^{d(N-1)}} 
		|u(\bx_1,\ldots,\bx_{j-1},\bx,\bx_{j+1},\ldots,\bx_N)|^2 \prod_{k \neq j} d\bx_k$.
\end{thm}

\begin{proof}
	We define for each 
	$x' = (\bx_1,\ldots,\bx_{j-1},\bx_{j+1},\ldots,\bx_N) \in \R^{d(N-1)}$ 
	a collection of functions
	$$
		\bx \ \mapsto \ 
		\phi_j(\bx,x') := u(\bx_1,\ldots,\bx_{j-1},\bx,\bx_{j+1},\ldots,\bx_N)
	$$
	in $L^2(Q)$,
	and proceed as in the proof of Theorem \ref{thm:single-particle-NLT}, 
	writing 
	$$
		\int_Q |\nabla_{\bx} \phi_j(\bx,x') |^2 \,d\bx 
		= \int_0^\infty \int_Q |\phi_j^{e,+}(\bx,x')|^2 \,d\bx \,de,
	$$
	for each $x' \in \R^{d(N-1)}$,
	and similarly to \eqref{LT-eigenvalue-estimate}
	\begin{multline*}
		\sum_{j=1}^N \int_{\R^{d(N-1)}} |\phi_j^{e,-}(\bx,x')|^2 \,dx' \\
		\le \sum_{j=1}^N \int_{\R^{d(N-1)}} \left( |Q|^{-1} + C_d e^{\frac{d}{2}} \right) 
			\int_Q |\phi_j(\by,x')|^2 \,d\by \,dx' \\
		= \left( |Q|^{-1} + C_d e^{\frac{d}{2}} \right) \int_Q \rho,
	\end{multline*}
	as well as 
	using the triangle inequality on $L^2(\R^{d(N-1)};\C^N)$,
	\begin{multline*}
		\left( \int_{\R^{d(N-1)}} \sum_{j=1}^N |\phi_j^{e,+}(\bx,x')|^2 \,dx' \right)^{\frac{1}{2}}
		\ge \left[ 
			\left( \int_{\R^{d(N-1)}} \sum_{j=1}^N |\phi_j(\bx,x')|^2 \,dx' \right)^{\frac{1}{2}} \right. \\
		\left. -
			\left( \int_{\R^{d(N-1)}} \sum_{j=1}^N |\phi_j^{e,-}(\bx,x')|^2 \,dx' \right)^{\frac{1}{2}}
		\right]_+,
	\end{multline*}
	for $\bx \in Q$. Hence,
	\begin{multline*}
		\sum_{j=1}^N \int_{\R^{d(N-1)}} \int_Q |\nabla_{\bx} \phi_j(\bx,x') |^2 \,d\bx \,dx' \\
		\ge \int_Q \int_0^\infty \left[ 
			\rho(\bx)^{\frac{1}{2}} - \left( (|Q|^{-1} + C_d e^{\frac{d}{2}}) \int_Q \rho \right)^{\frac{1}{2}}
			\right]_+^2 de \,d\bx,
	\end{multline*}
	and \eqref{box-LT-many} then follows as in \eqref{LT-integral}.
\end{proof}

\end{document}